\documentclass[leqno]{amsart}
\usepackage{caption}
\usepackage{verbatim}
\usepackage{amsmath,amssymb,amsthm,mathtools,mathrsfs,tikz}
\usepackage{environ}
\usepackage{stmaryrd}
\usetikzlibrary{arrows}
\usetikzlibrary{positioning}
\usetikzlibrary{decorations.text}
\usetikzlibrary{decorations.pathmorphing}
\usetikzlibrary{decorations.pathreplacing}
\makeatletter
\newsavebox{\@brx}
\newcommand{\llangle}[1][]{\savebox{\@brx}{\(\m@th{#1\langle}\)}%
	\mathopen{\copy\@brx\kern-0.5\wd\@brx\usebox{\@brx}}}
\newcommand{\rrangle}[1][]{\savebox{\@brx}{\(\m@th{#1\rangle}\)}%
	\mathclose{\copy\@brx\kern-0.5\wd\@brx\usebox{\@brx}}}
\makeatother
\makeatletter
\newsavebox{\measure@tikzpicture}
\NewEnviron{scaletikzpicturetowidth}[1]{%
	\mathrm{Def}\tikz@width{#1}%
	\mathrm{Def}\tikzscale{1}\begin{lrbox}{\measure@tikzpicture}%
		\BODY
	\end{lrbox}%
	\pgfmathparse{#1/\wd\measure@tikzpicture}%
	\edef\tikzscale{\pgfmathresult}%
	\BODY
}
\makeatother
\usepackage{bbm}
\DeclarePairedDelimiter\norm{\lvert}{\rvert}
\DeclarePairedDelimiter\inner{\langle}{\rangle}

\usepackage{bm}
\usepackage[colorlinks=true,linkcolor=blue,citecolor=blue]{hyperref}
\usepackage{enumitem}
\usepackage{csquotes}

\def\irr#1{{\rm  Irr}(#1)}

\numberwithin{equation}{section}
\newcounter{intro}

\newtheorem{introthm}[intro]{Theorem}

\newtheorem{thm}[equation]{Theorem}
\newtheorem{lem}[equation]{Lemma}
\newtheorem{cor}[equation]{Corollary}

\theoremstyle{remark}

\theoremstyle{definition}

\title[Nested Groups]{Groups where the centers of the irreducible characters form a chain II}

\author{Shawn T. Burkett}
\address{Department of Mathematical Sciences, Kent State University, Kent,
	Ohio 44240, U.S.A.} \email{sburket1@kent.edu}

\author{Mark L. Lewis}
\address{Department of Mathematical Sciences, Kent State University, Kent,
	Ohio 44240, U.S.A.} \email{lewis@math.kent.edu}

\date{\today}

\begin{document}

\begin{abstract}
In this paper, we determine new characterizations of nested and nested GVZ-groups, including character-free characterizations, but we additionally show that nested groups and nested GVZ-groups can be defined in terms of the existence of certain normal series.  
\end{abstract}

\maketitle

\section{Introduction}

All groups in this paper will be finite and when $G$ is a group, we will write $\irr G$ for the set of irreducible characters of $G$. In Problem \#24 of Research problems and themes I of \cite{YBGPPOV1}, Berkovich asks for a description of the $p$-groups $G$ for which the centers (quasi-kernels) of the irreducible characters form a chain with respect to inclusion.  In \cite{ML19gvz}, the second author did just this.  In this paper, we arrive at a different characterization of these groups using chains of subgroups that are defined for all groups.

Let $G$ be a group.  We will define a characteristic subgroup $K (G)$ in terms of the centers of certain irreducible characters of $G$.  Since the definition of this subgroup is technical, we postpone its statement until Section \ref{K section}.  We want to state some properties of this group.  

\begin{introthm}\label{intro Kprops}
Let $G$ be a nonabelian group, and let $N\lhd G$. Then either $K(G) \le N$ or $N \le Z(G)$ and $Z(G/N) = {Z}(G)/N$.
\end{introthm} 

If $G$ is a group, then we define $Z_2$ by $Z_2/Z(G) = Z(G/Z(G))$.  It turns out that the behavior of $K (G)$ depends on whether or not $Z_2 = Z(G)$ or $Z_2 > Z(G)$.  We next state the situation when $Z_2 = Z(G)$.

\begin{introthm} \label{intro Z2=Z}
Let $G$ be a nonabelian group.  Assume $Z_2 = Z(G)$.  Then the following are true:
\begin{enumerate}[label={\bf(\arabic*)}]
\item $K(G)$ is the intersection of all the noncentral normal subgroups of $G$.
\item $K(G) \not\le Z(G)$ if and only if $G$ has a unique subgroup $N$ that is minimal among noncentral normal subgroups of $G$.  	
\end{enumerate}   
\end{introthm}

When $Z_2 > Z(G)$, we have the following properties.

\begin{introthm}\label{intro Z_2>Z}
Let $G$ be a group and suppose $Z_2 > Z(G)$.  Then the following are true:
\begin{enumerate}[label={\bf(\arabic*)}]
\item  $K (G) \le Z (G)$ and $K (G)$ is an elementary abelian $p$-group for some prime $p$. 
\item If $K(G) > 1$, then $Z_2/Z(G)$ is a $p$-group where $p$ is the prime dividing $|K (G)|$.
\item $Z(G/K(G)) > Z(G)/K(G)$ if and only if every irreducible character $\chi \in \irr {G/K(G)}$ satisfies $Z (\chi) > Z(G)$.
\end{enumerate}
\end{introthm}

Following \cite{ML19gvz}, we say that a group $G$ is {\it nested} if for all characters $\chi, \psi \in \irr G$ either $Z (\chi) \le Z (\psi)$ or $Z (\psi) \le Z(\chi)$.  It is not difficult to see that a group $G$ is nested when the centers of the irreducible characters form a chain.  In \cite{SBML1}, we attach a subgroup to each conjugacy class, and we show that a group is nested if the subgroups attached to the conjugacy classes form a chain.  In this paper, we find another way to determine whether a group is nested.  Using the subgroup $K(G)$, we can determine when $G$ is nested.

\begin{introthm} \label{intro K > 1}
Let $G$ be a group.  Then $G$ is nested if and only if $K (G/N) > 1$ for every proper normal subgroup $N$. 
\end{introthm}

We will define a chain of subgroups by defining $K_0 = 1$ and for $i \ge 1$, we set $K_i$ by $K_i/K_{i-1} = K (G/K_{i-1})$.  Since $G$ is finite, this chain will terminate, and we write $K_\infty$ for the terminal term of this chain.
In particular, we can use this chain to determine if $G$ is nested.

\begin{introthm} \label{intro K_infty}
Let $G$ be a group.  Then $G$ is nested if and only if $K_\infty = G$.
\end{introthm}

When $G$ is nested, we will see that there is a correspondence between the groups that occur as the centers of irreducible characters and the $K_i$'s.  In \cite{ML19gvz}, the second author proved a number of results regarding the structure of the factors coming from the centers of characters.  We show that these results hold for the $K_i$'s even when $G$ is not nested.  We present one such result next.  When $N$ is a normal subgroup of $G$, we define $Z_N$ by $Z_N/N = Z(G/N)$. 

\begin{introthm} \label{intro z_k}
Let $G$ be a group.  Suppose that there exists integers $1 \le j < k$ so that $Z_{K_i} > Z_{K_{i-1}}$ and $[K_{i},G] \le K_{i-1}$ for all integers $i$ with $j \le i \le k$.   Then there is a prime $p$ so that $K_k/K_{j-1}$ and $Z_{K_k}/Z_{K_{j-1}}$ are $p$-groups.  In particular, $K_i/K_{i-1}$ and $Z_{K_i}/Z_{K_{i-1}}$ are elementary abelian $p$-groups for every integer $i$ with $j \le i \le k$.
\end{introthm}

Note that the $K_i$'s are an ascending series of normal subgroups.  Using the centers of characters in a different fashion, we construct a descending chain of normal subgroups $G = \delta_1 \ge \delta_2 \ge \dotsb\ge \delta_i$.  We let $\delta_\infty$ be the terminal term in this series.  We are able to show the following:

\begin{introthm} \label{intro delta}
Let $G$ be a group.  Then $G/[\delta_i,G]$ is a nested group for all integers $i$.  In particular, $G$ is nested if and only if $\delta_\infty = 1$.
\end{introthm}

The study of nested groups was initiated by Nenciu in \cite{AN12gvz} and \cite{AN16gvz} under the additional hypothesis of GVZ-groups.  A group $G$ is a {\it GVZ-group} if every irreducible character of $G$ vanishes off its center.  We note that the definition of nested in \cite{AN12gvz} and \cite{AN16gvz} is different than the definition used in this paper, but the second author has shown in \cite{ML19gvz} that the definition in this paper is equivalent to Nenciu's definition when $G$ is a GVZ-group.  

Using the vanishing-off subgroup for subgroups in \cite{MLvos09} and \cite{NM14} that generalizes the vanishing-off subgroup of a character, we define the subgroup $U (G)$.   To obtain results regarding $U (G)$, we will use results regarding Camina triples that were proved by Mlaiki in \cite{NM14}.  In particular, we will see that $U (G)$ is always a subgroup of $K (G)$ and often equals $K (G)$, and thus, we will see that $U (G)$ shares many of the properties of $K (G)$.  When $N$ is a normal subgroup of $G$, we define $\irr {G \mid N}$ to be the set of characters in $\irr G$ that do not have $N$ in their kernels.   One property that is different is the following:

\begin{introthm} \label{intro U prop}
Let $G$ be a group.  Then $U (G)$ is the largest subgroup of $G$ so that every character in $\irr {G \mid U(G)}$ is fully ramified with respect to $Z (G)$.
\end{introthm}

Using $U (G)$ in a manner similar to $K (G)$, we will be able to determine when $G$ is a nested GVZ-group.  In particular, we have the following theorem:

\begin{introthm} \label{intro U > 1}
Let $G$ be a group.  Then $G$ is a nested GVZ-group if and only if $U (G/N) > 1$ for every proper normal subgroup $N$ of $G$.
\end{introthm}

Finally, we will also use the vanishing-off subgroup to find a chain of subgroups $G = \epsilon_1 \ge \epsilon_2 \ge \dots \ge \epsilon_i$ so that $\epsilon_i \ge \delta_i$ and often $\epsilon_i = \delta_i$.  We write $\epsilon_\infty$ for the terminal element of this series.

\begin{introthm} \label{intro epsilon}
Let $G$ be a group.  Then $G/[\epsilon_i,G]$ is a nested GVZ-group for every integer $i$.  In particular, $G$ is a nested GVZ-group if and only if $\epsilon_\infty = 1$.
\end{introthm}

We close this section by noting that we initially defined the subgroup $U$ and the chain of $\epsilon$'s independently.  We used them to characterize nested GVZ-groups in our preprint \cite{SBML2}.  However, we realized that those definitions could be generalized, which led to the characterizations of nested groups in terms of the subgroup $K$ and the chain of $\delta$'s. We then obtained the results for $U$ and the chain of $\epsilon$'s as consequences of the more general work.

\section{The subgroup $K(G)$}\label{K section}

Let $G$ be a group. Set $\mathcal{X} = \{ \chi \in \irr G \mid {Z} (\chi) > {Z} (G) \}$. Observe that if $\lambda$ is a linear character of $G$, then $Z(\lambda) = Z(G)$.  Hence, if $G$ is abelian, then $\mathcal {X}$ is empty.  On the other hand, if $G$ is nonabelian, then $Z (G) < G$ and $Z(1_G) = G$ and so, $1_G \in \mathcal {X}$.  It follows that $\mathcal {X}$ is empty if and only if $G$ is abelian.   When $G$ is abelian, we will set $K(G) = G$.   When $G$ is nonabelian, define $K (G) = \bigcap_{\psi \in \mathcal{X}} \ker(\psi)$.  Thus, $K (G)$ is the intersection of the kernels of these characters.  We begin by showing when $G$ is nonabelian that $K (G)$ is always contained in the derived subgroup.

\begin{lem}\label{G'}
If $G$ is a nonabelian group, then $K (G) \le G'$.
\end{lem}

\begin{proof}
Since $G$ is nonabelian, we have $Z(G) < G$.  Recall that if $\lambda \in \irr {G/G'}$, then $Z(\lambda) = G > Z(G)$.  This implies that $\irr {G/G'} \subseteq \mathcal {X}$ and so, we have $G' = \bigcap_{\lambda \in \irr {G/G'}} \ker (\lambda) \ge K(G)$.  
\end{proof}

Recall that a group $G$ is quasi-simple if $G$ is perfect (i.e. $G' = G$) and $G/Z(G)$ is nonabelian simple.  We now show that abelian and quasi-simple groups are the only groups where $K (G) = G$.

\begin{lem}\label{K = G}
Let $G$ be a group.  Then $K (G) = G$ if and only if $G$ is abelian or $G$ is quasi-simple.
\end{lem}

\begin{proof}
By definition, $K (G) = G$ if $G$ is abelian.  Now suppose $G$ is quasi-simple and consider $1_G \ne \chi \in \irr G$.  Then $\ker (\chi) \le Z(G)$, and since $G/Z(G)$ is nonabelian simple, it follows that $Z(G/\ker (\chi)) = Z(G)/\ker (\chi)$.  Hence, $Z(\chi) = Z(G)$, and so, $\mathcal {X} = \{ 1_G \}$.  It follows that $K (G) = \ker (1_G) = G$.  On the other hand, suppose that $G$ is not abelian and not quasi-simple.  By Lemma \ref{G'}, we have $K(G) \le G'$.  If $G' < G$, then we have the conclusion; so, we may assume $G$ is perfect.  However, since $G$ is not quasimple, we know that $G/Z(G)$ is not simple, so there must exist a normal subgroup $M$ so that $Z(G) < M < G$.  If $\mu \in \irr {G/M}$, then $Z(G) < M \le Z(\mu)$.  It follows that $\irr {G/M} \subseteq \mathcal {X}$ and hence, $K (G) \subseteq \bigcap_{\mu \in \irr {G/M}} \ker (\mu) = M < G$.  This proves the lemma. 
\end{proof}

The next observation is immensely useful. In fact, it allows us to define $K(G)$ in a different way, without mention of characters. 

\begin{lem}\label{center kernel}
Let $G$ be a group and let $H \le G$.  If $\chi \in \irr G$, then $[H,G] \le \ker (\chi)$ if and only if $H \le Z(\chi)$.
\end{lem}

\begin{proof}
Assume that $[H,G] \le \ker(\chi)$. Then we have the quotient: 
$$
H \ker(\chi)/\ker(\chi) \le Z (G/\ker(\chi)) = Z (\chi)/\ker (\chi),
$$ 
and so it follows that $H \le Z(\chi)$. The converse is clear, since $Z (G/\ker(\chi)) = Z(\chi)/\ker (\chi)$.
\end{proof}

Let $G$ be a group.  Define $\gamma_G(g) = \{ [g,x] \mid x \in G\}$ for every element $g\in G$.  We let $[g,G]$ denote the subgroup generated by $\gamma_G(g)$.  


The following lemma appeared in \cite{SBML}.  Since the proof is only one line, we include it again here. 

\begin{lem} \label{cen cond}
Let $G$ be a group.  Fix an element $g \in G$ and a character $\chi \in \irr G$.  Then $g \in Z(\chi)$ if and only if $[g,G] \le \ker (\chi)$.
\end{lem}

\begin{proof}
This follows immediately from the definition that states: $Z(\chi)/\ker (\chi) = Z (G/\ker (\chi))$.
\end{proof}

We now obtain a different characterization of $\mathcal {X}$ and this yields a description of $K(G)$ in terms of the groups $[g,G]$

\begin{lem}\label{K intersection}
Let $G$ be a nonabelian group. Then $\mathcal {X} = \bigcup_{g \in G \setminus Z(G)} \irr {G/[g,G]}$, and thus, $\displaystyle K (G) = \bigcap_{g \in G\setminus Z(G)} [g,G]$.
\end{lem}

\begin{proof}
We work to prove the first conclusion.  Suppose $g \in G \setminus Z(G)$ and $\chi \in \irr {G/[g,G]}$.  By Lemma \ref{cen cond}, this implies that $g \in Z(\chi)$, and so, $Z(\chi) > Z(G)$.  This leads to the containment: 
$\bigcup_{g \in G \setminus Z(G)} \irr {G/[g,G]} \subseteq \mathcal {X}$.  Conversely, if $\chi \in \mathcal {X}$, then there exists some element $g \in Z(\chi) \setminus Z(G)$.  Applying Lemma \ref{cen cond}, we see that $\chi \in \irr {G/[g,G]}$.  This implies that $\mathcal {Z} \subseteq \bigcup_{g \in G \setminus Z(G)} \irr {G/[g,G]}$, and the first conclusion is proved.  

For each element $g \in G$, we have that $[g,G] = \bigcap_{\chi \in \irr {G/[g,G]}} \ker (\chi)$.  Hence, $\bigcap_{g \in G \setminus Z(G)} [g,G] = \bigcap_{g \in G \setminus Z(G)} \bigcap_{\chi \in \irr {G/[g,G]}} \ker(\chi)$.  Using the claim from the first paragraph, we have $\bigcap_{g \in G \setminus Z(G)} \bigcap_{\chi \in \irr {G/[g,G]}} \ker(\chi) = \bigcap_{\chi \in \mathcal {X}} \ker (\chi) = K(G)$.  This yields $K (G) = \bigcap_{g \in G \setminus Z(G)} [g,G]$, as desired.
\end{proof}

\section{Properties of $K(G)$}

We now survey some interesting properties of $K(G)$.  Notice that this includes Theorem \ref{intro Kprops}.

\begin{thm}\label{Kprops}
Let $G$ be a nonabelian group. Set $K = K(G)$.  Then the following hold: 
\begin{enumerate}[label={\bf(\arabic*)}]
\item If $N$ is normal in $G$, then either $K \le N$ or $N \le Z(G)$.  
\item If $N \lhd G$ and $K \nleq N$, then $Z_N/N = {Z} (G/N) = {Z}(G)/N$. 
\item If $N \lhd G$ and $K \nleq N$, then $KN/N \le K (G/N)$. 
\item If $Z_K = Z(G)$, then $K(G/K) = 1$.
\end{enumerate}
\end{thm}

\begin{proof}
We begin by noting that $K \nleq N$ implies $N \le Z(G)$ if $N$ is a normal subgroup of $G$.  Indeed if $K \nleq N$, then by the definition of $K$ there exists a character $\chi \in \irr{G/N}$ satisfying $Z(G) = Z(\chi)$.  Since $N \le Z(\chi)$, this implies that $N \le Z(G)$ proving (1).
	
Now, let $N$ be a normal subgroup of $G$ satisfying $K \nleq N$, and assume that ${Z} (G/N) = Z_N/N > {Z}(G)/N$. Consider a character $\chi \in \irr {G/N}$.  Since $N \le \ker(\chi)$, it follows that ${Z}(\chi) \ge Z_N > {Z}(G)$. In particular, $K \le \ker(\chi)$.  Hence, we have $\irr{G/N} \subseteq \irr{G/K}$, from which it follows that $K \le N$. This contradiction establishes (2).
	
Next, we prove (3). Assume that $N \lhd G$ and $K \nleq N$. Then we have $Z_N = Z(G)$ by (2). So by Lemma~\ref{K intersection}, we have $K (G/N) = \bigcap_{g \in G \setminus Z_N} [gN,G/N]$ and $K = \bigcap_{g \in G \setminus Z(G)} [g,G]$.  Hence, $K \le [g,G]$ for all $g \in G \setminus Z(G)$, and thus, $KN/N \le \bigcap_{g \in G \setminus Z(G)}[g,G]N/N$.  Notice that $\{aN \mid a \in \gamma_G (g) \} \subseteq \gamma_{G/N} (gN)$ for each $g \in G \setminus Z_N$ and $[g,G] \le N$ for each $g \in Z_N \setminus Z(G)$.  It follows that $\bigcap_{g \in G \setminus Z(G)} [g,G]N/N \le \bigcap_{g \in G \setminus Z_N} [gN,G/N]$, and this gives the containment in (2).  

Finally, we prove (4).  We are assuming that $Z_K = Z(G)$, which implies that $Z(G/K) = Z(G)/K$.  Following the definition of $\mathcal {X}$, we set $\mathcal {X}_K = \{ \chi \in \irr {G/K} \mid Z(\chi)/K > Z(G/K) \}$.  We claim that $\mathcal {X} = \mathcal{X}_K$.  Suppose that $\chi \in \irr {G/K}$ and $Z (\chi)/K > Z(G/K) = Z(G)/K$, and viewing $\chi$ as a character in $\irr G$, we have $Z (\chi) > Z(G)$.  Conversely, suppose $\chi \in \irr G$ with $Z (\chi) > Z(G)$.  By the definition of $K (G)$, we have $K (G) \le \ker (\chi)$ and $Z(G/K) = Z(G)/K \le Z(\chi)/K$.  This implies that $\chi \in \irr {G/K}$.  This prove the claim.  Now, applying the definition of $K (G)$, we have $K (G) = \bigcap_{\chi \in \mathcal {X}} \ker (\chi)$.  It follows that $K(G/K) = \bigcap_{\chi \in \mathcal {X}_K} \ker (\chi)/K = \bigcap_{\chi \in \mathcal {X}} \ker (\chi)/K = K(G)/K = 1$.
 \end{proof}

\begin{lem} \label{direct prods}
Let $M$ and $N$ be groups.
\begin{enumerate}[label={\bf(\arabic*)}]
\item If $M$ and $N$ are both nonabelian, then $K (M \times N) = 1$.
\item If $M$ is nonabelian and $N$ is abelian, then $K (M \times N) = K(M)$.
\end{enumerate}
\end{lem}

\begin{proof}
We will let $G = M \times N$.  Suppose $M$ and $N$ are both nonabelian.  Then there exist elements $m \in M \setminus Z (M)$ and $n \in N \setminus Z(N)$.  It follows that $1 \ne [m,G] \le M$ and $1 \ne [n,G] \le N$.  By Lemma \ref{K intersection}, we have that $K (G) \le [m,G] \cap [n,G] \le M \cap N  = 1$.  Now, suppose that $M$ is nonabelian and $N$ is abelian.  Given $g \in G \setminus Z(G)$, we have $g = mn$ for some $m \in M$ and $n \in N$.  Observe that if $m \in Z(M)$, then $g \in Z(G)$, so we have $m \not\in Z(M)$.  Observe that $[g,G] = [mn,MN] = [m,M]$ since $N$ is central in $G$.  It follows that $\bigcap_{g \in G \setminus Z(G)} [g,G] = \bigcap_{m \in M \setminus Z(M)} [m,M]$, and applying Lemma \ref{K intersection}, we have $K (G) = K (M)$.
\end{proof}

Let $G$ be a group. We let $Z_i$ denote the $i^{\rm th}$ member of the upper central series. That is, we set $Z_0 = 1$ and define $Z_i$ recursively by $Z_{i+1}/Z_i = Z(G/Z_i)$.  It is not difficult to see that $Z_{i+1} = \{ g \in G \mid [g,G] \le Z_i\}$ for ever integer $i \ge 0$.  We note that $G$ is nilpotent if and only if $G = Z_n$ for some integer $n$.  We now study $K (G)$ in the situation where $Z_2 = Z(G)$.  This theorem includes Theorem \ref{intro Z2=Z}.

\begin{thm}
Let $G$ be a nonabelian group.  Assume $Z_2 = Z(G)$.  Then the following are true:
\begin{enumerate}[label={\bf(\arabic*)}]
\item $K(G)$ is the intersection of all the noncentral normal subgroups of $G$.
\item $K(G) \not\le Z(G)$ if and only if $G$ has a unique subgroup $N$ that is minimal among noncentral normal subgroups of $G$.  In this case, $K (G) = N$ and $K(G) Z(G)/Z (G)$ is the unique minimal normal subgroup of $G/Z (G)$.  Also, $K (G) = [n,G]$ for some noncentral element $n \in G$.
\end{enumerate}   
\end{thm}

\begin{proof}
Observe that every noncentral normal subgroup contains a subgroup that is minimal among the noncentral normal subgroups of $G$.  It is not difficult to see that this implies that the intersection of the noncentral normal subgroups of $G$ equals the intersection of the minimal noncentral normal subgroups of $G$.  Hence, it suffices to prove that the intersection of the minimal noncentral normal subgroups of $G$ equals $K (G)$.

Let $N$ be minimal noncentral normal subgroup of $G$.  Since $N$ is noncentral, there exists an element $n \in G \setminus Z(G)$.  Because $n$ is noncentral, we have $[n,G] > 1$ and as $Z_2 = Z(G)$, we see that $[n,G] \not\le Z(G)$.  On the other hand, $[n,G]$ is normal in $G$ and $[n,G] \le N$.  The minimality of $N$ implies that $N = [n,G]$.  

Now, if $g \in G \setminus Z(G)$, then $[g,G]$ will be a noncentral normal subgroup of $G$ since $Z_2 = Z(G)$.  It follows that $[g,G]$ contains a minimal noncentral normal subgroup of $G$.  Let $\mathcal {N} = \{ n \in G \setminus Z(G) \mid [n,G] {\rm ~is~minimal~noncentral~normal}\}$.  It follows that if $g \in G \setminus Z(G)$, then there exists $n \in \mathcal {N}$ so that $[n,G] \le [g,G]$.  It is not difficult to see that this implies that $\bigcap_{g \in G \setminus Z(G)} [g,G] = \bigcap_{n \in \mathcal {N}} [n,G]$ is the intersection of the minimal noncentral normal subgroups of $G$.  Applying Lemma \ref{K intersection}, we obtain (1).

Suppose $K (G) \not\le Z(G)$.  Then there exists $N \le K(G)$ so that $N$ is a minimal noncentral normal subgroup of $G$. Let $M$ be any minimal noncentral normal subgroup of $G$.  By (1), we have $K (G) \le M$.  Notice that we now have $N \le K (G) \le M$.  Notice that this implies that $K (G) = N = M$, and hence, $N$ is the unique minimal noncentral normal subgroup of $G$.  Notice that $K(G) Z(G) > Z (G)$, so $K (G) Z (G)/Z (G) > 1$.  Let $L/Z(G)$ be a minimal normal subgroup of $G/Z(G)$.  Then $L$ is a noncentral normal subgroup of $G$.  Because $K (G)$ is the unique minimal noncentral subgroup of $G$, we have that $K (G) \le L$, and so, $K (G) Z (G) \le L$.  Now, the fact that $L/Z(G)$ is minimal normal implies that $L = K(G) Z(G)$.  Since $L$ was arbitrary, it follows that $K(G) Z(G)/Z(G)$ is the unique minimal normal subgroup of $G/Z(G)$.  We apply the earlier paragraph to see that $K (G) = [n,G]$ for any $n \in K(G) \setminus Z (G)$.  

Finally, suppose that $G$ has a unique minimal noncentral normal subgroup $N$.  Then $N$ is contained in every noncentral normal subgroup of $G$, and so, $N \le K (G)$.  On the other hand, by the second paragraph, we have $N = [n,G]$ for any $n \in N \setminus Z(G)$, and so, $N = [n,G] \ge \bigcap_{g \in G\setminus Z(G)} [g,G] = K (G)$ by Lemma \ref{K intersection}.  We conclude that $N = K(G)$ and $K (G) \not\le Z(G)$ as desired.
\end{proof}

We next consider the case where $K(G)$ is central.

\begin{lem}
Let $G$ be a nonabelian group.  If $K(G) > 1$ and $K(G) \le Z(G)$, then every minimal normal subgroup of $G$ is central.
\end{lem}

\begin{proof}
Let $N$ be a minimal normal subgroup of $G$.  By Lemma \ref{Kprops}, we have either $K(G) \le N$ or $N \le Z(G)$.  Since $N$ is minimal normal, if $K (G) \le N$, then $N = K (G)$ and we are assuming $K (G) \le Z(G)$.  This proves the result.
\end{proof}

We now see that when $Z_2 > Z(G)$ that $K (G)$ must be central and elementary abelian.  This includes Theorem \ref{intro Z_2>Z} (1) and (2).

\begin{thm}\label{Z_2 > Z}
Let $G$ be a group.  Then the following are true:
\begin{enumerate}[label={\bf(\arabic*)}]
\item If $Z_2 > Z(G)$, then $K (G) \le Z (G)$ and $K (G)$ is an elementary abelian $p$-group for some prime $p$. 
\item If $Z_2 > Z(G)$ and $K(G) > 1$, then $Z_2/Z(G)$ is a $p$-group where $p$ is the prime dividing $|K (G)|$.
\item If $Z_2 \ge Z_{K (G)} > Z(G)$, then $K(G)$ and $Z_{K (G)}/Z(G)$ are elementary abelian $p$-groups for some prime $p$ and $K(G) = [g,G]$ for all $g \in Z_{K(G)} \setminus Z(G)$.
\end{enumerate}
\end{thm}

\begin{proof}
Suppose that $Z_2 > Z(G)$.  Then there exists $g \in Z_2 \setminus Z(G)$.  It follows that $[g,G] \le Z(G)$.  By Lemma \ref{K intersection}, we have $K (G) \le [g,G] \le Z(G)$.  We can find a prime $p$ and element $x \in Z_2 \setminus Z(G)$ so that $x^p \in Z(G)$.  Notice that $[x,G] = \gamma_G (x)$ since $[x,G] \le Z(G)$.  Given an element $k \in K(G)$, there must exist an element $y \in G$ so that $k = [x,y]$.  Now, $k^p = [x,y]^p = [x^p,y] = 1$ since $x^p \in Z(G)$.  Since $k$ was arbitrary, this implies that $K(G)$ is an elementary abelian $p$-group.  This proves (1).
	
Suppose in addition that $K(G) > 1$.  Since $K(G) > 1$, we can find $1 \ne k \in K(G)$.  Now, consider any $g \in Z_2 \setminus Z(G)$, and observe that $[g,G] = \gamma_G (g)$ and $K (G) \le [g,G]$ by Lemma \ref{K intersection}.  Thus, we can find an element $h \in G$ so that $[g,h] = k$.  We see that $1 = k^p = [g,h]^p = [g^p,h]$.  Thus, $g^p$ centralizes $h$ and $g$ does not centralize $h$.  This implies that $p$ divides the order of $g Z(G)$ since otherwise $\langle g, Z(G) \rangle = \langle g^p,Z(G) \rangle \le C_G (h)$ and $g$ would centralize $h$ which is a contradiction.  This implies that $p$ divides the order of every nontrivial element of $Z_2/Z(G)$ and thus, $Z_2/Z(G)$ must be a $p$-group.  This implies (2).
	
Now, assume that $Z_{K (G)} > Z(G)$.  Notice that $Z_{K (G)} \le Z_2 (G)$ and $K (G) > 1$, so by (2), $K (G)$ is an elementary abelian $p$-group.  Let $g \in Z_{K(G)} \setminus Z(G)$.  Then $[g,G] \le K (G)$ and $[g,G] \le K(G)$ by Lemma \ref{K intersection}.  Hence, we have $[g,G] = K(G)$.  We see that $1 = [g,h]^p = [g^p,h]$ for all $h \in G$ and so, $g^p \in Z(G)$.  This implies that $Z_{K(G)}/Z(G)$ is an elementary abelian $p$-group.
\end{proof}

We will present an example a nonabelian $p$-group $G$ with $K (G) > 1$ where $Z_{K(G)} = Z (G)$.  Our next two results show some interesting properties that do hold when $Z_{K(G)} > Z(G)$. 

\begin{lem}\label{K equiv conditions}
Let $G$ be a group satisfying $K (G) > 1$ and $Z_2 > Z(G)$.  Write $K = K (G)$. Then the following are equivalent:
\begin{enumerate}[label={\bf(\arabic*)}]
\item $Z_{K} > Z(G)$. 
\item There exists an element $g \in G$ satisfying $[g,G] = K$. 
\item $K = [Z_{K},G]$. 
\end{enumerate}
\end{lem}

\begin{proof}
By Lemma \ref{Z_2 > Z} (3), we see that if $Z_K > Z(G)$, then $K = [g,G]$ for every element $g \in Z_K \setminus Z(G)$.  This shows (1) implies (2).  Suppose there is an element $g \in G$ so that $[g,G] = K$.  Then $g \in Z_K$ by the definition of $Z_K$.  This implies that $K = [g,G] \le [Z_K,G] \le K$, and so, $K = [Z_K,G]$.  We have (2) implies (3).  Since $K > 1$, we see that if $[Z_K,G] = K$, then $Z_K > Z(G)$.  Hence, (3) implies (1).
\end{proof}

We note that this next result strongly uses the hypothesis that $K(G) > 1$.  This includes Theorem \ref{intro Z2=Z} (3).   

\begin{lem} \label{K equiv cond 2}
Let $G$ be a group satisfying $K (G) > 1$ and $Z_2 > Z(G)$.  Write $K = K(G)$.  Then the following are equivalent:
\begin{enumerate}[label={\bf(\arabic*)}]
\item $Z_K > Z(G)$.
\item $K (G/N) = KN/N$ for every normal subgroup $N$ of $G$ not containing $K$. 
\item Every irreducible character $\chi \in \irr {G/K}$ satisfies $Z (\chi) > Z(G)$.
\end{enumerate}
\end{lem}

\begin{proof}
Assuming (1), we have that every character $\chi \in \irr {G/K}$ satisfies $Z(\chi) \ge Z_K > Z(G)$. Thus (3) holds.  We now show that (3) implies (2). Assume that $Z(\chi) > Z(G)$ for every character $\chi \in \irr {G/K}$, and assume to the contrary that there exists a normal subgroup $N$ not containing $K$ and satisfying $K (G/N) \ne KN/N$.  By Lemma \ref{Kprops} (2), we must have that $K(G/N) > KN/N$.  Then there exists a character $\psi \in \irr {G/N \mid K(G/N)}$ with $K$ in its kernel. Since $K$ is in the kernel of $\psi$, we have $Z(G) < Z(\psi)$ by our hypothesis.  On the other hand, since $K(G/N)$ is not contained in the kernel of $\psi$, viewed as a character of $G/N$, it must be that $Z(\psi)/N = Z(G/N)$.  Also, since $N$ is not contained in $K$, we use Lemma \ref{Kprops} (1) to see that $Z (G/N) = Z (G)/N$.  Combining these, we obtain $Z(G)/N < Z(\psi)/N = Z(G/N) = Z(G)/N$, which is a contradiction.  Hence, (2) holds.
	
Finally, we show that (2) implies (1). Assume that (1) is not true, and let $G$ be a counterexample of minimal order.  For every element $g \in Z_2$ that satisfies $g^p \in Z(G)$, it is not difficult to see that $[g,G]$ is elementary abelian.  By Lemma~\ref{K intersection}, we have $K \le [g,G]$.  So $[g,G]$ is elementary abelian and contains $K$, and it follows that $[g,G] = K \times C_g$ for some subgroup $C_g$.   Among the possible choices for $g$, we choose $g$ so that $|[g,G]|$ is minimal.
	
Since $G$ is a counterexample, we are assuming that $Z_K = Z(G)$.  In light of Lemma \ref{K equiv conditions}, we cannot have that $[g,G] = K$.  This implies that $[g,G] > K$, and hence, we must have $C_g > 1$.  Let $N \le [g,G]$ be a minimal normal subgroup of $G$ satisfying $K \nleq N$.  Note that $[g,G]/N = KN/N \times C_gN/N$.
	
We claim that $G/N$ satisfies (2). To that end, let $\overline{H}$ denote $HN/N$ for every subgroup $H \le G$, and suppose $\overline{H}$ is a normal subgroup of $\overline{G}$ not containing $K(\overline{G})$. Since $K(\overline{G}) = \overline{K}$, we deduce that $H$ does not contain $K$. So $K (\overline{G}/\overline{H}) \cong K(G/H) = KH/H \cong \overline{K}\overline{H}/\overline{H} = K(\overline{G})\overline{H}/\overline{H}$, as claimed.
	
Suppose $xN \in Z_{K (G/N)} \setminus Z(G/N)$.  By Lemma \ref{Z_2 > Z} (3), we have $[xN,G/N] = K(G/N)$.  On the other hand, using Lemma \ref{Kprops} (2), we obtain $Z(G/N) = Z(G)/N$.  Since $g \not\in Z(G)$, we see that $gN \not\in Z(G/N)$.  Using Lemma \ref{K intersection}, we deduce that $K(G/N) \le [gN,G/N]$.  Hence, $[xN,G/N] \le [gN,G/N]$.

Let $h \in G$.  Suppose $a \in \gamma_G (h)$, then $aN \in \gamma_{G/N} (hN)$.  We obtain $[h,G]N/N \le [hN,G/N]$.  Also, as $\gamma_G (g) = [g,G]$ contains $N$, this yields $\gamma_{G/N} (gN) = \{aN \mid a \in \gamma_G (g) \}$ and so, $[g,G]/N = [gN,G/N]$.  We now have $[x,G] \le [x,G]N \le [g,G]N = [g,G]$.  

Observe that $Z_{K (G/N)} \le Z_2 ((G/N)/(Z (G)/N)$.  It follows that $x \in Z_2$.  By Lemma \ref{Z_2 > Z}, we have that $(xN)^p \in Z(G/N)$ and so, $x^p \in Z(G)$.  By the minimality of $g$, we have $\norm {[g,G]} \le \norm {[x,G]}$.  This implies that $[x,G] = [g,G]$, and so, $[x,G]N/N \le K(G/N) = [xN,G/N] \le [gN,G/N] = [g,G]/N = [x,G]/N \le [x,G]N/N$.  We conclude that $K(G/N) = [g,G]/N$.

Recall that we are assuming $K (G/N) = KN/N$.  On the other hand, we have $K (G/N) = [g,G]/N = KN/N \times C_gN/N$.  This implies that $C_g \le N$ and since $N$ is minimal normal, $N = C_g$.  Certainly, $[g,G]$ contains a minimal normal subgroup $L$ different from $N$ and not containing $K$. The above argument shows that $C_g = L$, which is a contradiction. It follows that no such counterexample $G$ exists.
\end{proof}

Suppose $G$ is a group, if the character $\chi \in \irr G$ satisfies $Z (\chi) > Z(G)$, then the definition of $K (G)$ implies that $K (G) \le \ker (\chi)$.  Condition (2) of Lemma \ref{K equiv cond 2} implies that $Z (\chi) > Z(G)$ for every character $\chi \in \irr {G/K(G)}$ is equivalent to $Z_{K(G)} > Z(G)$.  This implies that if $Z_2 > Z(G)$ and $Z_K = Z(G)$, then there exists a character $\chi \in \irr {G/K}$ so that $Z(\chi)/K = Z(G)/K = Z(G/K)$.  Note that there exists a group $L$ so that $Z (\chi) > Z(L)$ for all $\chi \in \irr L$.  For example of such a group take $L$ to be  \verb+SmallGroup (32,27)+ in the computer algebra package Magma \cite{magma}.  Notice that if $G = L \times C_2$, where $C_2$ is the cyclic group of order 2, one obtains a group $G$ with a normal subgroup $N$ so that $Z(G/N) = Z(G)/N$ and every irreducible character $\chi \in \irr {G/N}$ satisfies $Z(\chi) > Z(G)$.  

\begin{lem}\label{cen int}
Let $N$ and $M$ be normal subgroups of $G$. Then $Z_{N \cap M} = Z_N \cap Z_M$.
\end{lem}

\begin{proof}
It is clear that $Z_{N \cap M} \le Z_N \cap Z_M$. Conversely, suppose $g \in Z_N \cap Z_M$, then $[g,G] \le N \cap M$.  It follows that $g \in Z_{N \cap M}$, and so, we have $Z_N \cap Z_M \le Z_{N\cap M}$ as well.
\end{proof}

As a corollary, we characterize $Z_{K (G)}$.

\begin{lem}\label{Z_K}
If $G$ is a nonabelian group, then $Z_{K(G)} = \bigcap_{\psi \in \mathcal{X}} Z(\psi)$.  
\end{lem}

\begin{proof}
By Lemma \ref{cen int}, we see that $\bigcap_{\psi \in \mathcal{X}} Z_{\ker(\psi)} = Z_{\bigcap_{\psi \in \mathcal{X}} \ker (\psi)} = Z_{K(G)}$.
\end{proof}

\section{Nested groups}\label{nested section}

This section is dedicated to characterizing nested groups in terms of $K$. Recall that a group is called {\it nested} if $\{ Z (\chi) \mid \chi \in \irr G\}$ is a chain with respect to inclusion.  In this case, we write $\{ Z (\chi) \mid \chi \in\irr G \} = \{ X_0, X_1,\dotsb, X_n \}$, where $G = X_0 > X_1 > \dotsb > X_n \ge 1$. Following \cite {ML19gvz}, we call this the {\it chain of centers} for $G$. We list here the results about nested groups from \cite{ML19gvz} that we require throughout this paper. 

\begin{lem}[{\normalfont\cite[Lemma 2.2, Corollary 2.5, Lemma 2.6]{ML19gvz}}]\label{lewisgvz}
Let $G$ be a nested group with chain of centers $G = X_0 > X_1 > \dotsb > X_n \ge 1$. Then the following statements hold: 
\begin{enumerate}[label={\bf(\arabic*)}]
\item $X_n = {Z} (G)$. 
\item $[X_i,G] < [X_{i-1},G]$ for each integer $1 \le i \le n$. 
\item Fix a character $\chi \in \irr G$.  Then ${Z} (\chi) = X_i$ if and only if $[X_i,G] \le \ker (\chi)$ and $[X_{i-1},G] \nleq \ker(\chi)$.
\end{enumerate}
\end{lem}

We begin by investigating the $K$-series.  This next lemma will be the key in showing that the series defined by $K (G)$ is connected to determining if $G$ a nested group.

\begin{lem}\label{nested K}
Let $G$ be nested with chain of centers $G = X_0 > X_1 > \dotsb > X_n \ge 1$. Then $K(G) = [X_{n-1},G]$.
\end{lem}

\begin{proof}
Observe that
\[\bigcap_{\substack{\chi \in \irr G \\ Z(\chi) > Z(G)}} Z(\chi) = X_{n-1} > X_n = Z(G),\]
where the last equality holds by Lemma~\ref{lewisgvz} (1).  We have $K (G) = \cap_{\chi \in \mathcal{X}} \ker (\chi)$.  Recall that $\chi \in \mathcal{X}$ if and only if $Z(\chi) > Z(G) = X_n$.  It follows that $\chi \in \mathcal {X}$ if and only if $X_{n-1} \le Z(\chi)$.  By Lemma \ref{center kernel}, we see that $X_{n-1} \le Z (\chi)$ if and only if $[X_{n-1},G] \le \ker (\chi)$.  It follows that $\mathcal X = \irr {G/[X_{n-1},G]}$, and since the intersection of the characters in $\irr {G/[X_{n-1},G]}$ will be $[X_{n-1},G]$, we conclude that $K (G) = [X_{n-1},G]$.
\end{proof}

We define $K_i$ recursively by setting $K_0 = 1$ and defining $K_{i+1}/K_i = K(G/K_i)$. Observe, using Lemma \ref {G'}, that $K_{i+1} \le G'$ if $K_i < G'$ and, by definition, that $K_{i+1} = G$ if $K_i = G'$.  Since $G$ is finite, there exists a positive integer $n$ so that $K_n = K_{n+1}$ and so, $k_i = K_n$ for all integers $i \ge n$.  Using this value of $n$, we define $K_\infty = K_n$.  This next result includes Theorem \ref{intro K_infty}.

\begin{thm}\label{kappa series}
Let $G$ be a group. Then $K_\infty = G$ if and only if $G$ is nested.  Moreover, if $G$ is nested with chain of centers $G = X_0 > X_1 > \dotsb > X_n \ge 1$, then $K_i = [X_{n-i},G]$ for every integer $i$ such that $0 \le i \le n$.
\end{thm}

\begin{proof}
First let $G$ be nested with chain of centers $G = X_0 > X_1 > \dotsb > X_n \ge 1$.  We show that $K_i = [X_{n-i},G]$ for each integer $i$ such that $0 \le i \le n$ by induction on $i$.  We have just established the case $i = 1$ in Lemma \ref{nested K}.  The case $i = 0$ follows from the fact that $X_n = Z(G)$ (see Lemma \ref{lewisgvz} (1)). So suppose $i \ge 1$, and assume that we have $K_i = [X_{n-i},G]$.  Since $G$ is nested, $G/K_i$ is also nested.   Because $Z(G/K_i) = X_{n-i}/K_i$, the chain of centers for $G/K_i$ is given by $G/K_i = X_0/K_i > \dotsb > X_{n-i}/K_i \ge K_i/K_i$ (see Lemma 2.7 of \cite{ML19gvz}). By Lemma~\ref{nested K}, we have that
\[K_{i+1}/K_i = K(G/K_i) = [X_{n-i-1}/K_i,G/K_i] = [X_{n-(i+1)},G]/K_i, \]
since $K_i = [X_{n-i},G] < [X_{n-i-1},G]$ (see Lemma \ref{lewisgvz} (2)). This gives $K_{i+1} = [X_{n-(i+1)},G]$, as desired.

Now we prove the converse by induction on $\norm{G}$.  By the inductive hypothesis, we have that $G/K(G)$ is nested. Since $K(G) \le \ker(\chi)$ for every character $\chi \in \irr G$ satisfying $Z(\chi) > Z(G)$, it follows that $\mathcal{C} = \{Z(\chi) \mid \chi\in\irr G\ \,\text{and}~ Z (\chi) > Z (G) \}$ is a chain. Hence, $\{Z (\chi) \mid \chi \in \irr G\} = \mathcal {C} \cup \{Z(G)\}$ is also a chain, which means that $G$ is nested. 
\end{proof}

As a corollary, we have Theorem \ref{intro K > 1}.

\begin{cor} \label{nested cor}
Let $G$ be a group.  Then $G$ is nested if and only if $K (G/N) > 1$ for every proper normal subgroup $N$ of $G$.
\end{cor}

\begin{proof}
Suppose $G$ is nested.  If $N$ is a normal subgroup of $G$, then $G/N$ is nested.  If $N$ is proper, then $G/N > 1$.  Notice that if $K (G/N) = 1$, then $K_\infty (G/N) = 1$ violating Theorem \ref{kappa series}.  Conversely, suppose that $K (G/N) > 1$ for every proper normal subgroup $N$.  We work by induction on $|G|$.  If $|G| = 1$, then $G$ is nested.  Thus, we may assume $G > 1$.  By hypotheses, this implies that $K(G) > 1$.  Note that $G/K(G)$ will satisfy the inductive hypothesis; so by induction we have that $G/K(G)$ is nested.  Using Theorem \ref{kappa series}, this implies that $K_\infty (G/K (G)) = G/K (G)$.  It is easy to see that this implies that $K_\infty = G$ and applying Theorem \ref{kappa series} once again, we have that $G$ is nested.
\end{proof}

We now prove that the factor groups for the $K$-series share the same properties that the factor groups for the chain of centers of a nested group have, even when the group is not nested.  We start with a lemma.

\begin{lem} \label{Z_{K_i}}
Let $G$ be a group.  If $Z_{K_i} > Z_{K_{i-1}}$ and $[K_i,G] \le K_{i-1}$ for some integer $i$, then $Z_2 (G/K_{i-1}) > Z(G/K_{i-1})$.
\end{lem}

\begin{proof}
We begin by observing that $[K_i,G] \le K_{i-1}$ implies that $K_i \le Z_{K_{i-1}}$.  Notice that $Z_{K_{i-1}}/K_{i-1} = Z(G/K_{i-1})$ and $K (G/K_{i-1}) = K_i/K_{i-1}$, so $Z_{K_i} > Z_{K_{i-1}}$ implies that $Z_2 (G/K_{i-1}) \ge Z_{K(G/K_{i-1})} = Z_{K_i}/K_{i-1} > Z(G/K_{i-1})$. 
\end{proof}

This next result which includes Theorem \ref{intro z_k} is a generalization of Theorem 1.4 (2) of \cite{ML19gvz}. 

\begin{thm}
Let $G$ be a group.  Suppose that there exists integers $1 \le j < k$ so that $Z_{K_i} > Z_{K_{i-1}}$ and $[K_i,G] \le K_{i-1}$ for all integers $i$ with $j \le i \le k$.   Then there is a prime $p$ so that $K_k/K_{j-1}$ and $Z_{K_k}/Z_{K_{j-1}}$ are $p$-groups.  In particular, $K_i/K_{i-1}$ and $Z_{K_i}/Z_{K_{i-1}}$ are elementary abelian $p$-groups for every integer $i$ with $j \le i \le k$.
\end{thm}
	
\begin{proof}
In light of Lemma \ref{Z_{K_i}}, we may apply Lemma \ref{Z_2 > Z} (2) and (3) to see that there exists a prime $p_i$ so that $K_{i}/K_{i-1}$ and $Z_{K_i}/Z_{K_{i-1}}$ are elementary abelian $p_i$-subgroups and $Z_2 (G/K_{i-1})$ is a $p_i$-subgroup.  We need to show that the $p_i$'s are all equal.  Notice that it suffices to show that $p_{i+1} = p_i$ for every each $i$.  Thus, we fix an integer $i$.  Notice that we may replace $G$ by $G/K_{i-1}$.   Thus, we may assume that $i = 1$.  It suffices to show that $p_1 = p_2$.  By Lemma \ref{Kprops} (4), we see that $Z_{K} > Z_{K_0} = Z(G)$ implies that $K_2 > K_1$.  Since $Z_{K_2} > Z_{K_1}$ and $K_2 \le Z_{K_1}$, we have that $Z_2 (G/K_1) > Z(G/K_1)$, and so by Lemma \ref{Z_2 > Z}, we have $K_2/K_1 = K(G/K_1) \le Z(G/K_1) = Z_{K_1}/K_1$.  This implies that $K_2 \le Z_{K_1}$.  If $K_2 \not\le Z(G)$, then $K_2/(K_2 \cap Z(G)) \cong K_2 Z(G)/Z(G)$ is a nontrivial subgroup of $Z_{K_1}/Z(G)$.  This implies that $K_2/(K_2 \cap Z_1)$ is a $p_1$-group.  On the other hand, $K_2/K_1$ is a $p_2$-group and $K_1 \le K_2 \cap Z(G)$.  It follows that $p_2 = p_1$ in this case.  The other possibility is that $K_2 \le Z(G)$. In this case, we have $Z (G) < Z_{K_1} < Z_{K_2} \le Z_2$.  Since $Z_2/Z(G)$ is a $p_1$-group and $Z_{K_2}/Z_{K_1}$ is a $p_2$-group, it follows that $p_1 = p_2$.   
\end{proof}

We are not convinced that the hypothesis on conclusion (3) in this next lemma needs to be included.  We would not be surprised if it were still true without the extra hypotheses.

\begin{lem}
Let $G$ be a group. Suppose that there exists integers $1 \le j < k$ so that $Z_{K_{i}} > Z_{K_{i-1}}$ and $[K_i,G] \le K_{i-1}$ for all integers $i$ with $j \le i \le k$.  Then the following are true:
\begin{enumerate}[label={\bf(\arabic*)}]
\item $[Z_{K_i},G] K_{j-1} = K_i$ for each integer $i$ with $j \le i \le k$.
\item $[Z_{K_k},G] \le Z_{K_{j-1}}$ if and only if $[K_k,G] \le K_{j-1}$.
\item In the situation of (2), the exponent of $Z_{K_k}/Z_{K_{j-1}}$ equals the exponent of $K_k/K_{j-1}$.
\end{enumerate}   
\end{lem}

\begin{proof}
We prove (1) by induction.  In light of Lemma \ref{Z_{K_i}} we may apply Lemma \ref{K equiv conditions} in $G/K_{j-1}$ to see that $[Z_{K_j},G] K_{j-1}/K_{j-1} = [Z_{K_j}/K_{j-1},G/K_{j-1}] = K_j/K_{j-1}$.  This yields the equality $[Z_{K_j},G] K_{j-1} = K_j$.  Suppose for some $i$ with $j < i \le k+1$ that $[Z_{K_{i-1}},G] Z_{K_{j-1}} = K_{i-1}$.  Applying Lemma \ref{K equiv conditions} in $G/K_{i-1}$, we obtain $K_i = [Z_{K_i},G] K_{i-1} = [Z_{K_i},G] ([Z_{K_{i-1}},G] K_{j-1}) = [Z_{K_i},G] K_{j-1}$.  This proves (1).
	
By part (1), we have that $[Z_{K_k},G] K_{j-1} = K_k$.  Observe that $[Z_{K_k},G] K_{j-1} = K_k \le Z_{K_{j-1}}$ if and only if $K_k/K_{j-1} \le Z(G/K_{j-1})$.  And $K_k/K_{j-1} \le Z (G/K_{j-1})$ if and only if $[K_k,G] \le K_{j-1}$.  This proves (2).

Let $p^e$ be the exponent of $Z_{K_k}/Z_{K_{j-1}}$ and let $p^f$ be the exponent of $K_k/K_{j-1}$.  Suppose $x \in Z_{K_k}$; this implies that $x^{p^e} \in Z_{K_{j-1}}$ and $[x,g] \in K_k$ for all $g \in G$.  We have that $[x,g]^{p^e} = [x^{p^e},g] \in K_{j-1}$.  This implies that $e \le f$.  On the other hand, we see that $[x^{p^f},g] = [x,g]^{p^f} \in K_{j-1}$ for all $g \in G$ implies that $x^{p^f} \in Z_{K_k}$ and so, $f \le e$.  We conclude that $f = e$.
\end{proof}

In \cite{ML19gvz}, the second author shows that if $G$ is nested, then every member of the upper central series appears as the center of some irreducible character. The same turns out to be true about the lower central series. Moreover, the lower central series is a subseries of the $K$-series. 

\begin{lem}
Let $G$ be a nested group with chain of centers $G = X_0 > X_1 > \dotsb > X_n \ge 1$.  If $N$ is a normal subgroup of $G$, then $[N,G] = [X_i,G]$ for some integer $i$ with $0 \le i \ne n$. In particular, for every integer $i \ge 2$, there exists an integer $j$ with $0 \le j \le n$ such that $G_i = [X_j,G]$.
\end{lem}

\begin{proof}
By Lemma 3.1 of \cite{ML19gvz}, we have that $Z (G/[N,G]) = X_i/[N,G]$ for some integer $i$ with $0 \le i \le n$. So $[X_i,G] \le [N,G]$.  But we also have $N \le X_i$, so $[N,G] \le [X_i,G]$. The second statement follows easily from the first.
\end{proof}

Let $G$ be a group.  Following the literature, we will say that a {\it minimal class} of $G$ is a non-central conjugacy class of $G$ whose size is minimal among the noncentral conjugacy classes of $G$.  If $G$ is a nested $p$-group, we have the following consequence of Theorem~\ref{Z_2 > Z}, which determines the size of a minimal class of $G$. 

\begin{lem}\label{min breadth}
If $G$ is a nested $p$-group with chain of centers $G = X_0 > X_1> \dotsb > X_n > 1$, then a minimal class of $G$ has size $\norm {[X_{n-1},G]}$.
\end{lem}

\begin{proof}
Write $m = \min \{\norm {\mathrm{cl}_G (g)} \mid g \in G \setminus Z(G)\}$.  By Theorem~\ref{Z_2 > Z}, any element of $X_{n-1} \setminus X_n$ has class size $\norm {[X_{n-1},G]}$.  So $m \le \norm{[X_{n-1},G]}$. By \cite{LMM99}, there exists an element $g \in Z_2$ lying in a minimal class.  Since $\norm {\mathrm{cl}_G (g)} = \norm {[g,G]}$ and $[X_{n-1},G] = K (G) \le [g,G]$ by Lemma \ref{K intersection}, we have $m = \norm {[g,G]} \le \norm {[X_{n-1},G]} \le m$.  This proves that $\norm {[g,G]} = m$ as desired.
\end{proof}




The final theorem of this section shows that we can characterize nested $p$-groups by finding elements with certain properties. 

\begin{thm} \label{last sect 4}
If $G$ is a nonabelian $p$-group then the following statements are equivalent:
\begin{enumerate}[label={\bf(\arabic*)}]
\item $G$ is nested.
\item For every normal subgroup $N \le G$ with $G' \not\le N$, there exists an element $g_N \in G \setminus N$ depending on $N$ so that $\{\chi \in \irr {G/N} \mid Z(\chi) = Z_N \} = \{ \chi \in \irr {G/N} \mid g_N \notin Z(\chi)\}$. 
\item For every normal subgroup $N \le G$ with $G' \not\le N$, there exists a normal subgroup $L_N$ depending on $N$ so that $N < L_N \le G$ so that $\{ \chi \in \irr {G/N} \mid Z (\chi) = Z_N \} = \irr {G/N \mid L_N/N }$. 
\end{enumerate}
\end{thm}

\begin{proof}
Let $G$ be a nonabelian, nested $p$-group. We now show that (1) implies (2). Hence, we assume that $G$ is a nested group.  For each normal subgroup $N \le G$ with $G' \not\le N$, we see that $G/N$ is also a nonabelian, nested $p$-group.  Set $K_N/N = K(G/N)$ and observe that $K_N/N < G/N$.   Since $G/K_N$ is nontrivial nilpotent, we have $Z_{K_N} > K_N$ and if $G/K_N$ is nonabelian, then we will have $Z_2 (G/N) > Z (G/N)$.  Given a character $\chi \in \irr {G/N}$, we have that $Z(\chi) = Z_N$ if and only if $K_N \nleq \ker(\chi)$ by Lemma~\ref{K equiv conditions} (5). But by Lemma~\ref{K equiv conditions} (2), there exists an element $1 \ne g_N N\in G/N$ for which $[g_N,G]N/N = [g_N N,G/N] = K_N/N$.  Since $K_N \le \ker(\chi)$ if and only if $g_N \in Z(\chi)$, we have (2).  Since $[g_N,G]$ is a normal subgroup of $G$, we also have (3) by taking $L_N = [g_N,G]$. 

To complete the proof, it remains only to show that (3) implies (1).  Now assume (3); we prove that $G$ is nested by induction on $\norm{G}$.  We start by showing that $K(G) > 1$.  Indeed, there exists a nontrivial normal subgroup $L = L_1$ of $G$ so that $Z (\chi) = Z (G)$ if and only if $\chi \in \irr {G \mid L}$.  Thus, $L \le \ker(\chi)$ for every character $\chi \in \irr G$ satisfying $Z(\chi) > Z(G)$. It follows that $L \le K(G)$ (in fact $L = K(G)$), which implies that $K(G) > 1$. Observe that (3) holds for any quotient of $G$, so if $N > 1$ is normal in $G$, then $G/N$ is nested by the inductive hypothesis and so, $K(G/N) > 1$.  It follows from Corollary \ref{nested cor}  that $G$ is also nested.
\end{proof}

If $G$ is nested and nilpotent, then \cite[Corollary 4.10]{ML19gvz} gives that $G = P \times Q$, where $P$ is a nested $p$-group for some prime $p$ and $Q$ is an abelian $p'$-group. In this case, the $K$-series is determined from the $p$-part of $G$.

We note the following consequence of Corollary~4.10 of \cite{ML19gvz} and Lemma~\ref{direct prods}. If $G$ is a nested nilpotent group, then the $K$-series gives a central series for $G'$ with elementary abelian $p$-quotients, for a fixed prime $p$.  This fact was proved by the second author as \cite[Lemma 4.5]{ML19gvz} (using a different method).

\section{Camina triples and Vanishing-off subgroups}\label{vanishing section}

In this section, we introduce the subgroup $U (G)$, and we will see that $U(G)$ serves the same role for nested GVZ-groups as $K(G)$ serves for nested groups. We first define a subgroup $U (G \mid N)$ for every normal subgroup $N \le G$ that, in some sense, determines a set of characters of $G$ vanishing on $G \setminus N$. The subgroup $U (G)$ arises when $N = Z(G)$, and thus, $U(G)$ is related to fully ramified characters. 

The following observation is useful.

\begin{lem} \label{irr sets}
Let $M$ and $N$ be normal subgroups of $G$. Then $M \le N$ if and only if $\irr {G \mid M} \subseteq \irr {G \mid N}$.
\end{lem}

\begin{proof}
It is clear that $\irr {G\mid M} \subseteq \irr {G\mid N}$ when $M \le N$.  Conversely, suppose that $\irr {G\mid M} \subseteq \irr {G\mid N}$.  This implies that $\irr {G/N} \subseteq \irr {G/M}$, and so 
$$
M = \bigcap_{\chi \in \irr {G/M}} \ker (\chi) \le \bigcap_{\chi \in \irr {G/N}} \ker(\chi) = N.
$$
\end{proof}

We begin with a review of vanishing-off subgroups of characters.  Recall that the vanishing-off subgroup $V(\chi)$ of a character is defined by 
$$
V(\chi) = \langle g \in G \mid \chi (g) \ne 0 \rangle.
$$ 
(This subgroup is defined on page 200 of \cite{MI76}.)  In particular, $V (\chi)$ is the smallest subgroup $V$ such that $\chi$ vanishes on $G \setminus V$.   The second author extends this definition to groups in \cite{MLvos09} by defining $V(G)$ to be the subgroup defined by 
$$
V (G) = \langle g \in G \mid \chi (g) \ne 0 {\rm ~for~some~} \chi \in \irr G \rangle.
$$
It is not difficult to see that $V (G)$ is the smallest subgroup of $G$ so that every nonlinear character in $\irr G$ vanishes on $G \setminus V(G)$.  

Let $N$ be a normal subgroup of $G$. Following \cite{NM14}, we define ${V} (G \mid N)$ by
$$
V (G \mid N) = \langle g \in G \mid \chi (g) \ne 0 {\rm ~for~some~} \chi \in \irr {G \mid N} \rangle.
$$ 
Thus, ${V} (G \mid N)$ is the smallest subgroup $V$ of $G$ such that every character $\chi \in \irr {G \mid N}$ vanishes on $G \setminus V$.  Observe that $V(G \mid G') = V(G)$.  Note that if $N = 1$, then this product is empty and we follow the convention above that the empty product yields the trivial subgroup, so $V (G \mid 1) = 1$.  We will need the following properties.

\begin{lem}\label{Vprops}
The following statements hold for every pair $H$ and $N$ of normal subgroups of $G$.
\begin{enumerate}[label={\bf(\arabic*)}]
\item $N \le {V} (G \mid N)$.
\item $V (G \mid N) = \prod_{\chi \in \irr {G \mid N}} V (\chi)$.
\item ${V} (G \mid HN) = {V} (G \mid H) {V} (G \mid N)$.
\item If $N \le H$, then ${V} (G \mid N) \le {V} (G \mid H)$.
\end{enumerate}
\end{lem}

\begin{proof}
To see (1), suppose that there exists an element $n \in N \setminus {V} (G \mid N)$.  Then $\chi (n) = 0$ for every character $\chi \in \irr {G \mid N}$, and $n$ lies in the kernel of every other irreducible character of $G$.  By column orthogonality (e.g., see Theorem 2.18 of \cite{MI76}), one sees that $\norm {N} = \norm {\mathrm {cl}_G (n)}$, which is strictly less than $\norm {N}$. Thus, no such element $n$ exists.

We next show (2). Observe that $V (\chi) \le V (G \mid N)$ for all characters $\chi \in \irr {G \mid N}$.  It follows that $\prod_{\chi \in \irr {G \mid N}} V(\chi) \le V (G \mid N)$.  Conversely, if $g \in G$ satisfies $\chi (g) \ne 0$ for some character $\chi \in \irr {G \mid N}$, then $g \in V (\chi)$ and so, $g \in \prod_{\chi \in \irr {G \mid N}} V(\chi)$.  It follows that the generators of $V( G \mid N)$ all lie in $\prod_{\chi \in \irr {G \mid N}}$; so, $V (G \mid N) \le \prod_{\chi \in \irr {G \mid N}} V(\chi)$ as desired.

We now show (3). To accomplish this, we show that $\irr {G \mid HN} = \irr {G \mid H} \cup \irr {G \mid N}$.  First, fix a character $\chi \in \irr {G \mid HN}$.  If $\chi \in \irr {G \mid H}$, then we have desired result.  Thus, we may assume that $H \le \ker (\chi)$.  Since $HN \not\le \ker(\chi)$, we cannot have $N$ contained in $\ker (\chi)$, so we must have $\chi \in \irr {G \mid N}$. The reverse containment is obvious. Thus, we have by (2) that ${V} (G \mid HN) = \prod_{\chi \in \irr {G\mid HN}} {V}(\chi)$ and ${V} (G \mid H) {V} (G \mid N) = \prod_{\chi \in \irr {G \mid H}} {V}(\chi) \prod_{\chi \in \irr {G \mid N}} {V}(\chi)$.  The first observation implies that these products are equal.

Finally, (4) is immediate from (2) and the fact that $\irr {G \mid N} \subseteq \irr {G \mid H}$ whenever $N \le H$ by Lemma \ref{irr sets}. 
\end{proof}

For each normal subgroup $N$ of $G$, define ${U} (G \mid N) = \prod_{H \in \mathcal{H}} H$ where $\mathcal {H} = \{ H \lhd G \mid {V} (G \mid H) \le N \}$.  Notice that if $N = G$, then $\mathcal {H}$ will be all normal subgroups of $G$ and so, $U (G \mid G) = G$.  We will show when $N < G$ that $U = {U} (G \mid N)$ is the largest normal subgroup of $G$ for which every member of $\irr {G \mid U}$ vanishes on $G \setminus N$. In particular, if $H \nleq U$, then there exists a character $\chi \in \irr {G \mid H}$ that does not vanish on $G \setminus N$. Therefore, the subgroup ${U} (G \mid N)$ identifies a set of characters that, in some sense, is maximal with respect to vanishing on $G \setminus N$.


\begin{lem}\label{uiff}
Let $H$ and $N$ be normal subgroups of a group $G$. Then $H \le {U} (G \mid N)$ if and only if ${V} (G \mid H) \le N$. 
\end{lem}

\begin{proof}
If ${V} (G \mid H)\le N$, then it is clear that $H \le {U} (G \mid N)$.  Conversely, if $H \le {U} (G \mid N)$, then ${V} (G \mid H) \le {V} (G \mid {U} (G \mid N)) =  \prod_{K \in \mathcal {V}}  {V} (G \mid K) \le N$ where $\mathcal {V}  = \{ K \lhd G \mid {V} (G \mid K) \le N\}$.
\end{proof}

Lemma~\ref{uiff} implies that the maps $N \mapsto {U} (G \mid N)$ and $N \mapsto {V} (G \mid N)$ give a (monotone) Galois connection from the lattice $\mathrm{Norm} (G)$ of normal subgroups of $G$ to itself. For more information on Galois connections, we refer the reader to \cite{galoisprimer}.

We now present some basic properties of ${U} (G \mid N)$.

\begin{lem}\label{uproperties}
Let $G$ be a nonabelian group.  The following hold. 
\begin{enumerate}[label={\bf(\arabic*)}]
\item For each normal subgroup $N \le G$, the subgroup ${U} (G \mid N)$ is the unique largest subgroup, $U \le G$, such that every character in $\irr {G \mid U}$ vanishes on $G \setminus N$.
\item For each normal subgroup $N \le G$ and every element $g \in G$, we have $g \in {U} (G \mid N)$ if and only if every character $\chi \in\irr G$ satisfying $g \notin \ker(\chi)$ vanishes on $G \setminus N$.
\item For each normal subgroup $N \le G$, we have ${U} (G \mid N) \le N \cap G'$. 
\item If $N$ is characteristic in $G$, so is ${U} (G \mid N)$.
\end{enumerate}
\end{lem}

\begin{proof}
If every character in $\irr {G \mid H}$ vanishes on $G \setminus N$, then ${V} (G \mid N) \le H$, so $H \le {U} (G \mid N)$ by Lemma~\ref{uiff}. This establishes (1).

To show (2), first note that for an irreducible character $\chi \in \irr G$, we have $g \in \ker(\chi)$ if and only if $\inner{g}^G \le \ker (\chi)$, where $\inner {g}^G$ denotes the normal closure of $\inner {g}$. Hence,  every character $\chi \in \irr G$ satisfying $g \notin \ker (\chi)$ vanishes on $G \setminus N$ if and only if every character $\chi \in \irr {G \mid \inner {g}^G}$ vanishes on $G \setminus N$. The latter happens if and only if ${V} (G \mid \inner{g}^G) \le N$, which happens if and only if $\inner {g}^G \le {U} (G \mid N)$. Finally, we note since ${U} (G \mid N)$ is normal in $G$ that ${U} (G \mid N)$ contains $g$ if and only if it contains $\inner {g}^G$.

Next, we prove (3). The fact that $U (G \mid N) \le N$ follows from Lemma~\ref{uiff} as $N \le V (G \mid N)$ by Lemma~\ref{Vprops}. The rest of statement (3) follows from the fact that no linear character can vanish on any element of $G$. In particular, this means that $\irr {G \mid U (G \mid N)} \subseteq \irr{G \mid G'}$.

Since $N$ uniquely determines $U (G \mid N)$, it follows that if $N$ is characteristic in $G$, then $U (G \mid N)$ will be characteristic.
\end{proof}

We next consider how the map $N \mapsto {U} (G \mid N)$ interacts with quotients.

\begin{lem}\label{uquotients}
Let $H$ and $N$ be normal subgroups of $G$ that satisfy ${V} (G \mid N) \le H$. Then ${U} (G/N \mid H/N) = {U} (G \mid H)/N$.
\end{lem}

\begin{proof}
Let $x \mapsto \overline{x}$ denote the canonical surjection $G \to G/N$. Define the sets 
\[\mathcal {C} = \{K \lhd G \mid {V} (G \mid K) \le H\}\ \,\text{and}\ \,\mathcal{D} = \{\overline {K} \lhd \overline {G} \mid {V} (\overline {G} \mid \overline {K}) \le \overline {H}\}.\] 
We claim that $\mathcal {D} = \overline {\mathcal {C}}$. However, we first show that $\mathcal {D} = \overline {\mathcal {C}'}$, where
\[\mathcal{C}' = \{K \lhd G \mid N \le K\ \,\text{and}\ \,{V} (G \mid K) \le H\}.\] 
To that end, let $K \lhd G$ satisfy $N \le K$. Then $N \le {V} (G \mid K)$, and so we have
\[{V} (\overline {G} \mid \overline{K}) {V} (G \mid N)/N = {V} (G \mid K)/N.\]
So it follows that $\overline{K} \in \mathcal{D}$ if and only if $K \in \mathcal{C}'$, as claimed. 

In particular, this gives
\[{U} (\overline{G} \mid \overline{H}) = \prod_{K \in \mathcal{D}} K = \prod_{K \in \mathcal{C}'} \overline {K} = \overline {\prod_{K \in \mathcal{C}'} K}.\]

Next note that since ${V} (G \mid N) \le H$, we have $K \in \mathcal{C}$ if and only if $KN \in \mathcal{C}'$. Hence
\[{U} (G \mid H) = \prod_{K \in \mathcal{C}} K = \prod_{K \in \mathcal{C}} KN = \prod_{K \in \mathcal{C}'} K.\]
The result now follows by taking quotients.
\end{proof}

We see that the subgroups $V (G \mid N)$ and $U (G \mid N)$ are both related to zeros of irreducible characters.  This next result was proved in \cite{SBML} and instead gives insight on which characters a specific element of $G$ vanishes.

\begin{lem} \label{basics}
Let $M$ be a normal subgroup of $G$ and let $g \in G \setminus M$.  Then the following are equivalent: 
\begin{enumerate}[label={\bf(\arabic*)}]
\item $g$ is conjugate to every element in $gM$. 
\item For every element $z \in M$, there exists an element $x \in G$ so that $[g,x] = z$. 
\item $|C_G (g)| = |C_{G/M} (gM)|$. 
\item $\chi (g) = 0$ for all $\chi \in \mathrm{Irr}(G \mid M)$. 
\end{enumerate}
\end{lem}

If every element $g \in G \setminus M$ satisfies the equivalent conditions of Lemma~\ref{basics}, we call $(G,M)$ a {\it Camina pair}.  These objects were first considered by A. Camina in \cite{acamina}, as a natural generalization of Frobenius groups. If there is a normal subgroup $N$ of $G$ containing $M$ so that every element $g \in G\setminus N$ satisfies the equivalent conditions of Lemma~\ref{basics}, we call $(G,N,M)$ a {\it Camina triple}. These objects were first studied by Mattarei in his Ph.D. thesis \cite{SM92}. Many more properties of Camina triples were found by Mlaiki in \cite{NM14}, where the following result appears.

\begin{lem}[{\normalfont \cite[Theorem 2.1]{NM14}}] \label{equiv}
Let $M$ and $N$ be normal subgroups of $G$.  Then the following are equivalent: 
\begin{enumerate}[label={\bf(\arabic*)}]
\item $(G,N,M)$ is a Camina triple. 
\item For every element $g \in G \setminus N$ and every element $z \in M$, there exists an element $x \in G$ so that $[g,x] = z$. 
\item $|C_G (g)| = |C_{G/M} (gM)|$ for all $g \in G \setminus N$. 
\item $V (G \mid M) \le N$. 
\item $\chi (g) = 0$ for every element $g \in G \setminus N$ and for every character $\chi \in \irr {G \mid M}$. 
\end{enumerate}
\end{lem}

Since, by Lemma~\ref{uiff}, we have that $V (G \mid M) \le N$ if and only if $M \le U (G \mid N)$, Lemma~\ref{equiv} yields the following result.

\begin{lem}\label{U camina}
Let $M$ and $N$ be normal subgroups of a group $G$.  Then $(G,N,M)$ is a Camina triple if and only if $M \le {U} (G \mid N)$. In particular, $(G,N)$ is a Camina pair if and only if $N = {U} (G \mid N)$.
\end{lem}

In the next section, we study the subgroup $U (G) = U (G \mid Z(G))$. It turns out that this subgroup is closely related to the subgroup $K(G)$ defined earlier. The next result gives an alternate description of $U (G \mid N)$ that, once specialized to the case $N = Z(G)$, foreshadows the connections between $K(G)$ and $U(G)$.

\begin{lem}\label{Ukernels}
Let $N < G$ be a normal subgroup.  Then $U (G \mid N) = \bigcap\limits_{\chi \in \mathcal {U}} \ker(\chi)$ where $\mathcal {U} = \{ \chi \in \irr G \mid V (\chi) \nleq N\}$. 
\end{lem}

\begin{proof}
Write $U = U (G \mid N)$ and let $W = \bigcap_{\chi \in \mathcal {U}} \ker(\chi)$.  Note that $V (G \mid U) \le N$, and so $V (\chi) \le N$ for all characters $\chi \in \irr {G \mid U}$.  This implies that $U \le \ker(\chi)$ for every character $\chi \in \mathcal {U}$.  We deduce that $U \le W$. Conversely, consider a character $\chi \in \irr {G \mid W}$.  Since $W \nleq \ker(\chi)$, we must have $V (\chi) \le N$.  In particular, $\chi$ vanishes on $G \setminus N$.  It follows that $W \le U$, and hence also that $W = U$.
\end{proof}

\section{The subgroup $U (G)$}\label{U section}

Define ${U} (G) = {U} (G \mid {Z} (G))$.  When $G$ is an abelian group, we have $Z(G) = G$ and so $U (G) = G$.  We will study $U (G)$ when $G$ is nonabelian.  Notice that Theorem \ref{intro U prop} follows from Lemma \ref{uproperties} in this case.

In this section, we will study the properties of the subgroup $U (G)$, many of which are shared by the subgroup $K(G)$. We will use these subgroups to tailor the results of $K(G)$ regarding nested groups specifically to nested GVZ-groups. Before doing this, we show that this subgroup is useful in describing existing properties of groups related to GVZ-groups. 

A group $G$ is called a {\it VZ-group} if every nonlinear irreducible character vanishes on $G \setminus Z (G)$. These groups were studied in Fern\'{a}ndez-Alcober and Moret\'{o} in \cite{VZ01} and also by the second author in \cite{MLIGT08}.  
We note that both nested groups and GVZ-groups generalize VZ-groups. 

\begin{lem}
Let $G$ be a group.  The following are equivalent: 
\begin{enumerate}[label={\bf(\arabic*)}]
\item $G$ is a VZ-group. 
\item ${Z} (G) = {V} (G)$. 
\item ${U} (G) = G'$. 
\end{enumerate}
\end{lem}

\begin{proof}
Both (2) and (3) are equivalent to the condition that every nonlinear character $\chi \in \irr G$ vanishes off of $Z(G)$, which is equivalent to $G$ being a VZ-group.
\end{proof}

An intimately related concept is a semi-extraspecial group. A $p$-group $G$ is called {\it extraspecial} if $G' = Z(G)$ is the socle of $G$, and $G$ is called {\it semi-extraspecial} if $G/N$ is extraspecial for every maximal subgroup $N$ of $Z(G)$.  In particular, if $G$ is a VZ-group such that $Z (G) = G'$, then $G$ is a semi-extraspecial group.  

We showed in Lemma~\ref{Kprops} that $K(G) \le G'$, and we showed in Lemma~\ref{Z_2 > Z} that $K(G)\le Z(G)$ when $Z_2>Z(G)$. Our first result about $U (G)$ shows that the same is true of $U (G)$.  We actually show a bit more so that we may deduce a different description of semi-extraspecial groups. 

\begin{lem} \label{properties}
Let $G$ be a nonabelian group.  Then ${U} (G) \le G' \cap {Z} (G) \le G' {Z} (G) \le {V}(G)$.
\end{lem}

\begin{proof}
We have $U (G) \le G' \cap Z (G)$ by Lemma~\ref{uproperties} (2). Since $Z(G) \le V (\chi)$ for all characters $\chi \in \irr G$, we have $Z (G) \le V (G)$, and so, also $G' Z(G) \le V (G)$.
\end{proof}

\begin{lem}
If $G$ is a nontrivial group, then ${U} (G) = {V} (G)$ if and only if $G$ is a semi-extraspecial group.
\end{lem}

\begin{proof}
Assume ${U} (G) = {V} (G)$.  Then $G$ is not abelian, since if it were, we would have $U (G) = G > 1 = V (G)$. In light of Lemma \ref{properties}, we have ${U} (G) = {Z} (G) = G' = {V} (G)$. This means that $(G, {Z} (G))$ is a Camina pair, so $G$ is a $p$-group for some prime $p$, obviously of nilpotence class $2$.  Since ${Z} (G) = G'$, we see that $G$ is in fact a Camina group.  It is known that being a Camina group of nilpotence class $2$ is equivalent to being semi-extraspecial (e.g. see \cite[Theorem 1.2]{LVses}).

Suppose $G$ is a semi-extraspecial group; then $G' = {Z} (G)$ and every nonlinear irreducible character of $G$ vanishes on $G \setminus {Z} (G)$.  Then ${V} (G) \le {Z} (G)$ and $G' \le {U} (G)$. Since the reverse containments always hold, we have ${U}(G) = G' = {Z}(G) = {V}(G)$.
\end{proof}

We next find several descriptions of $U (G)$, the last of which was alluded to at the end of Section~\ref{vanishing section}. We must however introduce some new notation. For each element $g \in G \setminus Z(G)$, we let $D_G (g)/Z(G) = C_{G/Z(G)} (gZ(G))$; i.e. $D_G (g) = \{x \in G \mid [g,x] \in Z(G)\}$.  

\begin{lem}\label{UDg}
Let $G$ be a group. The following hold.
\begin{enumerate}[label={\bf(\arabic*)}]
\item ${U} (G) = \bigcap\limits_{g \in G \setminus {Z} (G)} [g,D_G(g)]$.
\item $U (G)$ is the largest normal subgroup of $G$ contained in $\bigcap\limits_{g \in G \setminus Z(G)} \gamma_G (g)$.
\item $U (G) = \bigcap\limits_{\chi \in \mathcal{V}} \ker(\chi)$ where $\mathcal {V} = \{\chi \in \irr G \mid V(\chi) > Z(G) \}$.
\end{enumerate}
\end{lem}

\begin{proof}
We first note that if $G$ is abelian, then each of these intersections is $G$, as is $U (G)$. We therefore assume that $G$ is nonabelian. 
	
We begin by proving (1). Set $Z = {Z} (G)$, $U = {U} (G)$, and $D = \bigcap_{g \in G \setminus Z} [g,D_G(g)]$.  Suppose first that $N \le [g,D_G(g)]$ for all $g \in G \setminus Z$.  Let $\chi \in \irr {G \mid N}$.  Take $\lambda$ to be the unique irreducible constituent of $\chi_Z$ and set $\nu = \lambda_N$.  Observe that $\nu \ne \mathbbm{1}_N$.  We see that $N \le [g,D_G(g)]$ for every element $g \in G \setminus Z$, that $\ker(\mu) < N$, and that $\ker(\mu) = \ker (\lambda) \cap N$. If there exists an element $g \in G \setminus Z$ so that $[g, D_G (g)] \le \ker(\lambda)$, then $N = [g, D_G (g)] \cap N \le \ker(\lambda) \cap N = \ker(\mu) < N$ which is a contradiction.  Thus, we have $[g, D_G (g)] \nleq \ker(\lambda)$ for all $g \in G \setminus Z$. By Lemma 3.3 of \cite{SBML}, $\lambda$ (and hence $\chi$) is fully ramified with respect to $G/Z$.  Since $U$ is the unique largest subgroup of $G$ so that all characters in $\irr {G \mid U}$ vanish on $G \setminus Z$, we have $N \le U.$  Taking $N = D$, we obtain $D \le U$. If $U = 1$, then we must have $D = 1$ also. So, we may assume that $U > 1$. Let $g \in G \setminus Z$. Since $(G,Z,U)$ is a Camina triple, we have $gU \subseteq \mathrm{cl}_G(g)$. In particular, for each element $u\in U$, there exists an element $x_u \in G$ so that $[g, x_u]=u$. Since $x_u \in D_G(g)$, we have $U \le [g, D_G (g)]$. It follows that $U \le D$.
	
Now, we prove (2). Write $S = \bigcap_{g \in G \setminus Z(G)} \gamma_G(g)$. Let $H$ denote the largest normal subgroup of $G$ contained in $S$. Note that $U(G) \subseteq \gamma_G (g)$ for every $g\in G\setminus Z(G)$, and so $U(G)\le H$. Suppose that $N$ is a normal subgroup of $G$ contained in $S$. Then $gN \subseteq \mathrm {cl}_G (g)$ for every element $g\in G \setminus Z(G)$. By Lemma~\ref{basics}, this means that every character $\chi \in \irr{G \mid N}$ vanishes off of $Z(G)$, so $N \le U(G)$. It follows that $H\le U(G)$.  Finally, statement (3) is just Lemma~\ref{Ukernels} with $N = Z(G)$.
\end{proof}

\begin{lem}\label{U le K} 
Let $G$ be a group. Then ${U} (G) \le K (G)$.  
\end{lem}

\begin{proof}
The result is trivially true if $G$ is abelian, so assume that this is not the case. The fact that $U (G) \le K(G)$ follows immediately from any of the statements in Lemma~\ref{UDg}.  

\end{proof}

\begin{cor}\label{U eq K}
If $Z_{U(G)} > Z(G)$, then $U (G) = K (G)$.
\end{cor}

\begin{proof}
By Lemma~\ref{U le K}, we know that $U(G) \le K(G)$. If $K(G) \nleq U(G)$, then $Z(G/U(G)) = Z(G)/U(G)$ by Lemma~\ref{Kprops} (2).
\end{proof}

Combining Lemma \ref{U le K} with Lemma \ref{direct prods}, we obtain the following:

\begin{lem}
Suppose $M$ and $N$ are groups. Then the following hold:
\begin{enumerate}[label={\bf(\arabic*)}]
\item If $M$ and $N$ are nonabelian, then $U (M \times N) = 1$.
\item If $M$ is nonabelian and $N$ is abelian, then $U (M \times N) = U (M)$.
\end{enumerate}
\end{lem}

\begin{proof}
By Lemma \ref{U le K}, we have $U (M \times N) \le K (M \times N)$.  Applying Lemma \ref{direct prods} (1), we have $K(M \times N) = 1$, and it follows that $U (M \times N) = 1$ when $M$ and $N$ are nonabelian.  When $M$ is nonabelian, and $N$ is abelian, we take $G = M \times N$.  Observe that $\gamma_G (m,n) = \gamma_M (m)$ for all $m \in M$ and $n \in N$.  It follows that $\cap_{(m,n) \in G \setminus Z(G)} \gamma_G (m,n) = \cap_{m \in M \setminus Z(M)} \gamma_M (m)$.  Conclusion (2) now follows from Lemma \ref{UDg} (2). 
\end{proof}

We will now show that a group $G$ satisfying ${U}(G) > 1$ is essentially a $p$-group in the sense that it is a $p$-group up to a central direct factor. To do so, we will appeal to the connection to Camina triples mentioned earlier. The following result can be found in \cite{NM14}.

\begin{lem}\label{NM2}{\normalfont(cf. \cite[Theorem 2.10]{NM14})}\hspace{\labelsep}
Let $(G, N, M)$ be a Camina triple. If $G/N$ is not a $p$-group for any prime $p$, then $M \cap {Z}(G) = 1$.
\end{lem}

Observe that if ${U}(G) > 1$, then Lemma~\ref{U camina} tells us that $(G,{Z}(G),{U}(G))$ is a Camina triple. It turns out that the structure of $G$ is quite restrictive in this case.

\begin{lem}\label{nontrivueqpq}
Let $G$ be a nonabelian group that satisfies ${U}(G) > 1$.  Then the following are true:
\begin{enumerate}[label={\bf(\arabic*)}]
\item $(G,Z(G),U(G))$ is a Camina triple.
\item $G = P \times Q$, where $p$ is a $p$-group for some prime $p$ and $Q$ is an abelian $p'$-group. In particular, $G$ is nilpotent.
\item $U (G)$ is an elementary abelian $p$-group for some prime $p$.
\end{enumerate} 
\end{lem}

\begin{proof}
Since ${V} (G \mid {U} (G)) \le {Z}(G)$, we have that $(G, {Z}(G), {U}(G))$ is a Camina triple by Lemma~\ref{U camina}. By Lemma~\ref{NM2}, $G/{Z}(G)$ must be a $p$-group since $1 < {U} (G) \le {Z} (G)$. Thus, if $Q$ is a complement for a Sylow $p$-subgroup of ${Z}(G)$, then $Q$ is direct factor of $G$.  This proves (1).  To prove (2), we have that $U (G) \le K (G)$ by Lemma \ref{U le K}.  Since $G$ is nonabelian and nilpotent, we have $Z_2 > Z(G)$.  Thus, we may apply Lemma \ref{Z_2 > Z} to see that $U (G)$ is an elementary abelian $p$-group.
\end{proof}

In \cite{MLvos09}, the second author shows that if $G$ is nonabelian, nilpotent and satisfies ${V} (G) < G$, then $G = P\times Q$, where $P$ is a $p$-group and $Q$ is an abelian $p'$-group. In particular, Lemma~\ref{nontrivueqpq} may also be considered an analog of that result.

We now continue to study the subgroup $U (G)$. The next result shows that two more properties from Lemma~\ref{Kprops} satisfied by $K(G)$ are shared by $U (G)$.  We will present examples to see that $U (G) > 1$, but $Z_{U(G)} = Z(G)$ does occur.

\begin{lem}\label{centerproperty}
Let $N \lhd G$, and assume that $U = {U} (G) \nleq N$. Then the following hold:  
\begin{enumerate}[label={\bf(\arabic*)}]
\item $N \le Z(G)$.
\item ${Z} (G/N) = {Z} (G)/N$.
\item $UN/N \le {U} (G/N)$. In particular, ${U} (G/N)$ is nontrivial. 
\item If $Z_U = Z (G)$, then $U (G/U) = 1$.
\end{enumerate}
\end{lem}

\begin{proof}
By Lemma \ref{U le K}, we have $U  \le K (G)$.  Thus, if $U \not\le N$, then $K (G) \not\le N$.  Now, (1) and (2) follow immediately from Lemma \ref{Kprops} (1) and (2).  The proof of (3) follows the same lines of the proof Lemma \ref{Kprops} (3) where we substitute $[g, D_G (g)]$ for $[g,G]$.  In view of Lemma \ref{UDg} (4), we can prove (4) along the same lines as the proof of Lemma \ref{Kprops} (4) where the property $Z(\chi) > Z(G)$ is replaced by $V(\chi) > Z(G)$.
\end{proof}

It therefore suffices for our considerations to consider only $p$-groups when discussing groups $G$ satisfying $U(G) > 1$. We saw earlier that some nice properties hold whenever $Z_{K(G)} > Z(G)$. We now show that similar properties hold whenever $Z_{U(G)} > Z(G)$. In fact, we obtain a slightly stronger result than Lemma~\ref{min breadth} for $p$-groups satisfying $Z_{U(G)} > Z(G)$.  

The following result will be useful; it shows that even more properties of $U(G)$ are mirrored by properties of $K(G)$.  

\begin{lem}\label{equivconjs}
Let $G$ be a nonabelian $p$-group satisfying $U = {U}(G) > 1$.  The following are equivalent: 
\begin{enumerate}[label={\bf(\arabic*)}]
\item ${Z}_U > {Z}(G)$. 
\item There exists an element $g \in G$ satisfying $[g,G] = U$. 
\item $U = [Z_U,G]$. 
\end{enumerate}
\end{lem}

\begin{proof}
Suppose (1).  By Corollary \ref{U eq K}, we have that $U = K (G)$, and using Lemma \ref{K equiv conditions}, we have (2).  Next, suppose (2).  In light of Lemma \ref{K intersection}, we have $K (G) \le [g,G] = U$, and by Lemma \ref{U le K}, we have $K (G) \le U(G)$, and so, $K (G) = U$, and $Z(G) < Z_U = Z_{K(G)}$ and applying Lemma \ref{K equiv conditions}, we have $U = K(G) = [Z_{K(G)},G] = [Z_U,G]$.   The fact that (3) implies (1) is immediate.
\end{proof}

Let $G$ be a group.  Recall that a {\it minimal class} of $G$ is a non-central conjugacy class of $G$ whose size is minimal among the noncentral conjugacy classes of $G$.  

\begin{lem}\label{mann sub}
Let $G$ be a $p$-group. If $Z_{U(G)} > Z(G)$, then every minimal class of $G$ lies in $Z_2$ and has size $\norm{U(G)}$. Moreover, we have
\[Z_{U(G)} = \{g \in G \mid \norm{\mathrm{cl}_G(g)} \le \norm{U(G)} \} = Z(G) \cup \{g \in G \mid \mathrm{cl}_G (g) = g U (G) \}.\]
\end{lem}

\begin{proof}
Let $m$ be the size of a minimal class of $G$. The assumption that $Z_{U (G)} > Z (G)$, implies $\norm {U(G)} > 1$.  By Lemma \ref{nontrivueqpq} (1), we know that $(G,Z(G),U(G))$ is a Camina triple, and by Lemma \ref{basics}, we see that every non-central class of $G$ is a union of $U(G)$-cosets, so $m$ divides the size of every non-central conjugacy class of $G$. Since $Z_{U(G)} > Z(G)$, there exists an element $g \in G \setminus Z(G)$ satisfying $[g,G] \le U(G)$ by Lemma \ref{equivconjs}.  Since $U (G) \le [g,G]$ by Lemma \ref{UDg}, we have $U (G) = [g,G]$ and hence $\mathrm{cl}_G (g) = gU(G)$. It follows that $m = \norm{U(G)}$. Thus, if $x \in G$ satisfies $\norm {\mathrm{cl}_G(x)} = m$, then $\mathrm{cl}_G(x) = xU(G)$. This means that if $x$ belongs to a minimal class, then $[x,G] = U(G)$. In particular, this means $x$ lies in $Z_2$ if $x$ belongs to a minimal class. Now, let $H = \{g \in G: \norm{\mathrm{cl}_G (g)} \le m \}$, and let $g, h \in H \setminus {Z}(G)$. Then $[gh,G] \le [g,G][h,G] = U$. It follows that $[gh,G]$ is either $U(G)$ or trivial, and so $gh \in H$.
\end{proof}

We mention one consequence of Lemma~\ref{mann sub}. In \cite{mann1}, Mann considers the subgroup generated by the minimal elements, and he proved when $G$ is a $p$-group, that this subgroup has nilpotence class at most $3$.  In \cite{isma}, Isaacs considered the subgroup $M (G)$ generated by the minimal elements and the central elements.  He proved that if $G$ is supersolvable, then $M (G)$ is nilpotent of nilpotence class at most $3$.  In \cite{mann2}, Mann further generalized Isaacs's result.  In the situation of Lemma~\ref{mann sub}, we have that $Z_{U(G)}$ is the {\it Mann subgroup} of $G$, and its nilpotence class is at most $2$.

\begin{lem} \label{U equiv}
Let $G$ be a nonabelian $p$-group and let $U = U (G) > 1$.  The following are equivalent:
\begin{enumerate}
\item $Z_{U(G)} > Z(G)$.
\item $U(G/N) = UN/N$ for every normal subgroup $N$ of $G$ not containing $U$. 
\item An irreducible character $\chi$ vanishes off of ${Z}(G)$ if and only if $\chi \in \irr {G \mid U}$.
\end{enumerate}
\end{lem}

\begin{proof}
Assume (1) holds. Since ${Z}(G/U) > {Z}(G)/U$, we have that ${Z}(\chi) > {Z}(G)$ for every character $\chi \in \irr {G/U}$. So if $\chi \in \irr {G/U}$, we must have ${V} (\chi) > {Z}(G)$.   This implies that no character in $\irr {G/U}$ vanishes off of $Z (G)$.  On the other hand, the definition of $U (G)$ implies that the characters in $\irr {G \mid U}$ all vanish off of $Z (G)$.  Thus (3) follows.

Next, assume (3), and suppose that ${U} (G/N) > UN/N$. As $\irr {G/N\mid UN/N} = \irr {G\mid U} \cap \irr {G/N}$, there exists a character $\chi \in \irr {G/UN}$ that vanishes off of ${Z} (G/N)$. By Lemma~\ref{centerproperty}, ${Z} (G/N) = {Z} (G)/N$, which implies that $\chi \in \irr {G/U}$ and vanishes off of ${Z}(G)$. As this is a contradiction, (2) holds.

Finally, the proof that (2) implies (1) follows the same lines as the proof that (2) implies (1) in Lemma \ref{K equiv cond 2} where we use $U$ in place of $K$. 
\end{proof}

We now show that when $G$ is a GVZ-group, $U (G)$ and $K (G)$ are equal.   

\begin{lem}\label{U = K}
Let $G$ be a GVZ-group. Then ${U} (G) = K (G)$. In particular, this holds if $G$ has nilpotence class 2.
\end{lem}

\begin{proof}
If $K(G) \nleq \ker(\chi)$ for some character $\chi \in \irr G$, then ${Z} (\chi) = {Z} (G)$, so $\chi$ is fully ramified over ${Z}(G)$ since $G$ is a GVZ-group. This implies $K(G) \le {U}(G)$. The reverse inclusion comes from Lemma~\ref{U le K}.
\end{proof}

One may prove Lemma~\ref{U = K} in a variety of different ways. For example, the lemma follows directly from the definition of $K(G)$ and Lemma~\ref{UDg} (3), since $V (\chi) = Z (\chi)$ for all characters $\chi \in \irr G$ whenever $G$ is a GVZ-group. It also follows from a combination of Lemma~\ref{UDg} (2) and the fact that $\gamma_G (g) = [g,G]$ for all $g \in G$ when $G$ is a GVZ-group which is proved in \cite{SBML}.

We now define a chain of subgroups $U_i$ by $U_0 = 1$ and $U_{i+1}/U_i = {U} (G/U_i)$ for each $i \ge 0$. Recall from Lemma~\ref{properties} that $U (G) \le G'$ when $G$ is nonabelian. It follows that $U_{i+1} \le G'$ when $U_i < G'$.  On the other hand, we see that $U_{i+1} = G$ when $U_i = G'$.  Observe that Lemma~\ref{U = K} implies that $U_i = K_i$ for all integers $i$ if $G$ is a GVZ-group, and Lemma \ref{U le K} implies that we always have $U_i \le K_i$.  Furthermore, one can see from Corollary~\ref{U eq K} and Lemma~\ref{centerproperty} (4) that $U_i < K_i$ implies that $U_i = U_{i+1}$.  Note that we will write $U_\infty$ for the terminal term in this series.  

\begin{thm}\label{nested gvz U thm}
Let $G$ be a nonabelian group. Then $G$ is a nested GVZ-group if and only if $U_\infty = G$.  Moreover, in the event that $G$ is a nested GVZ-group with chain of centers $G = X_0 > X_1 > \dotsb > X_n > 1$, then $U_i = [X_{n-i},G]$ for all $i$. 
\end{thm}

\begin{proof}
If $G$ is a GVZ-group, then so is any epimorphic image of $G$. Also, we have that $U (G) = K (G)$ by Lemma \ref{U = K}.  Hence, $K_i = U_i$ for all $i$ when $G$ is a GVZ-group.  Thus, $G$ being a nested GVZ-group implying $G = U_\infty$ follows from Theorem~\ref{kappa series}.   If $G$ is a nested GVZ-group with chain of centers $G = X_0 > X_1 >\dotsb > X_n > 1$, then Theorem~\ref{kappa series} also gives $U_i = K_i = [X_{n-i},G]$.  

Suppose $U_\infty = G$.  Since $U_i \le K_i$, this implies $K_\infty = G$ and by Theorem \ref{kappa series}, this implies that $G$ is nested.  Since $U_\infty = G$, we may apply Lemma \ref{Z(chi) = V(chi)} to see that $G$ is a GVZ-group.   
\end{proof}

We now have the promised proof of Theorem \ref{intro U > 1}.

\begin{cor}
Let $G$ be a group.  Then $G$ is a nested GVZ-group if and only if $U (G/N) > 1$ for every proper normal subgroup $N$ of $G$.
\end{cor}

Notice that the proof of this will essentially be identical to the proof of Corollary \ref{nested cor} with Theorem \ref{nested gvz U thm} in place of Theorem \ref{kappa series}, so we do not repeat it.

We conclude this section by presenting the analog of Theorem \ref{last sect 4} to nested GVZ-groups.  Notice that we replace the condition that $Z(\chi) = Z_N$ with $\chi$ is fully ramified over $Z_N$ which we have seen is equivalent to $V (\chi) = Z (\chi) = Z_N$. 

\begin{thm}\label{nested gvz thm 3}
Let $G$ be a nonabelian group. The following are equivalent. 
\begin{enumerate}[label={\bf(\arabic*)}]
\item $G$ is a nested GVZ-group. 
\item For every normal subgroup $N \lhd G$, there exists an element $g_N \in G \setminus N$ so that $\{ \chi \in \irr {G/N} \mid \chi  {\rm~is~fully~ramified~over~} Z_N \} = \{ \chi \in \irr {G/N} \mid g_N \notin Z(\chi) \}.$ 
\item For every normal subgroup $N\lhd G$, there exists a normal subgroup $N < L_N \lhd G$ so that $\{ \chi \in \irr {G/N} \mid \chi  {\rm~is~fully~ramified~over~} Z_N \}  = \{ \chi \in \irr {G/N \mid L_N/N}\}.$
\end{enumerate}
\end{thm}

\begin{proof}
Assume that $G$ is a nested GVZ-group. Observe that every quotient of a nested GVZ-group is also a nested GVZ-group. Therefore, statement (2) will follow if we show that whenever $G$ is a nested GVZ-group, there exists a nonidentity element $g \in G$ so that $\chi$ is fully ramified over $Z(G)$ if and only if $g \not\in Z(\chi)$. To that end, note that $U(G) > 1$ by Theorem \ref{nested gvz U thm}. By that same theorem, we have $Z_{U(G)} > Z(G)$, for otherwise we would have $U_2 = U (G)$ by Lemma \ref{centerproperty} (3). Hence by Lemma~\ref{equivconjs} there exists an element $g\in G$ so that $U (G) = [g,G]$; we also know from Lemma~\ref{equivconjs} that a character $\chi \in \irr G$ is fully ramified over $Z(G)$ if and only if $U(G) \nleq\ ker(\chi)$. Since $U (G) = [g,G] \le \ker(\chi)$ if and only if $g \in Z(\chi)$, statement (2) follows.  Also, since $[g,G] \lhd G$, statement (3) follows.
	
Finally, assume that (3) holds. Then there exists a nontrivial normal subgroup $L$ of $G$ so that every character  $\chi \in \irr {G \mid L}$ is fully ramified over $Z(G)$.  This means that $L \le U (G)$. In particular, we have $U (G) > 1$. Note that (3) must hold for every quotient of $G$, hence for $G/U(G)$.  Proceeding by induction on $\norm{G}$, we have that $G/U(G)$ is a nested GVZ-group. Thus, by Theorem \ref{nested gvz U thm} we have that $U_\infty = G$, and so, $G$ is a nested GVZ-group.
\end{proof}

\section {The $\delta$--series and the $\epsilon$--series}
 
We see from Theorem~\ref{kappa series} that nonabelian nested groups can be defined by the existence of a certain ascending normal series, which is a central series whenever $G$ is nilpotent. We now show that nested groups can be characterized by the existence of a descending normal series, which is also central whenever $G$ is nilpotent. Define the subgroups $\delta_i$ by setting $\delta_1 = G$ and $\displaystyle \delta_{i+1} = \!\!\!\!\!\!\! \prod_{\chi \in \irr {G \mid [\delta_i,G]}} \!\!\!\!\!\!\!\!\! Z(\chi) $ for every integer $i\ge 1$.  Note that if $[\delta_i,G] = 1$ (i.e, $\delta_i \le Z(G)$), then $\irr {G \mid [\delta_i,G]}$ is empty, and follow the convention that the empty product is the trivial subgroup.

\begin{lem} \label{contain}
If $i\ge 1$ is an integer, then $\delta_{i+1}$ is a subgroup of $\delta_i$.
\end{lem}

\begin{proof}
We prove this by induction on $i$. It is clear that $\delta_2 \le \delta_1 = G$. Now assume for some integer $i \ge 1$ that $\delta_{i+1} \le \delta_i$. Then $\irr {G \mid [\delta_{i+1},G] } \subseteq \irr {G \mid [\delta_{i} , G]}$, and thus we have $\displaystyle \delta_{i+2} = \!\!\!\!\!\!\! \prod_{\chi \in \irr {G \mid [\delta_{i+1},G]}} \!\!\!\!\!\!\!\!\! Z(\chi) \le \!\!\!\!\!\!\! \prod_{\chi \in \irr {G \mid [\delta_{i},G]}} \!\!\!\!\!\!\!\!\! Z(\chi) = \delta_{i+1}$
as desired.
\end{proof}

We now show that the $\delta_i$'s are a second normal series that determines if a group is nested.   

\begin{lem}\label{nested basics}
Let $G$ be a group and let $i$ be a positive integer.  
\begin{enumerate}[label={\bf(\arabic*)}]
\item If $\mathcal{M}_i = \{ N ~{\rm is~normal~in~} G \mid [\delta_i,G] \not\le N \}$, then $\delta_{i+1} = \prod_{N \in \mathcal{M}_i} Z_N$.
\item If $N$ is a normal subgroup of $G$, then either $N \le Z_N \le \delta_{i+1}$ or $[\delta_i,G] \le N$.
\item If $N$ is a normal subgroup of $G$ and $i > 1$ is integer that satisfies $[\delta_i,G] \le N$ and $[\delta_{i-1},G] \not\le N$, then $Z_N = \delta_i$.
\item If $i > 1$ and $[\delta_i,G] < [\delta_{i-1},G]$, then $Z_{[\delta_i,G]} = \delta_i$.
\item $G/[\delta_i,G]$ is a nested group.
\end{enumerate}
\end{lem}

\begin{proof}
Observe that $\{ \ker (\chi) \mid \chi \in \irr {G \mid [\delta_i,G]} \} \subseteq \mathcal{M}_i$.  Hence, we see that $\delta_{i+1} \le \prod_{N \in \mathcal{M}_i} Z_N$.  On the other hand, if $N \in \mathcal{M}_i$, then $N < N[\delta_i,G]$ and so there exists a character $\chi \in \irr {G/N \mid N[\delta_i,G]/N}$.  This implies that $\chi_{N[\delta_i/G]}$ has a nonprincipal irreducible constituent $\nu$.  Since $N[\delta_i,G]/(N \cap [\delta_i,G]) = N/(N \cap [\delta_i,G]) \times [\delta_i,G]/(N \cap [\delta_i,G])$, we see that $\nu_{[\delta_i,G]}$ is irreducible and nonprincipal.  This implies that $\chi \in \irr {G \mid [\delta_i,G]}$.  It follows that $Z_N \le Z(\chi) \le \delta_{i+1}$.  We conclude that $\prod_{N \in \mathcal{M}_i} Z_N \le \delta_{i+1}$.  This gives the desired equality in (1).
	
To prove (2), let $N$ be a normal subgroup of $G$.  If $[\delta_i,G] \le N$, then we are done.  Otherwise, $[\delta_i,G] \not\le N$ and so $N \le Z_N \le \prod_{N \in \mathcal{M}_i} Z_N = \delta_{i+1}$.  This proves (2).  

We now work to prove (3).  Suppose $[\delta_i,G] \le N$ and $[\delta_{i-1},G] \not\le N$.  Since $[\delta_i,G] \le N$, we see that $\delta_i \le N$.  On the other hand, since $[\delta_{i-1},G] \not\le N$, we have by (2) that $N \le \delta_i$.  This proves $N = \delta_i$ and (3) is proved.  Notice that (4) is just (3) with $N = [\delta_i,G]$.  
	
To prove (5), suppose $N$ is a normal subgroup and $[\delta_i,G] \le N$.  If $G' \le N$, then $Z_N = G$.  Otherwise, we can find an integer $j \le i$ so that $[\delta_j,G] \le N$ and $[\delta_{j-1},G] \not\le N$.  We then apply (3) to see that $Z_N = \delta_j$.  It follows that $G/[\delta_i,G]$ is a nested group by Lemma 3.2 of \cite{ML19gvz}.  
\end{proof}

Again, since $G$ is finite, there exists a positive integer $n$ so that $\delta_n = \delta_{n+1}$.  We set $\delta_\infty = \delta_n$.  This next result is Theorem \ref{intro delta}.

\begin{thm}\label{delta series}
Let $G$ be a nonabelian group. Then $\delta_\infty = 1$ if and only if $G$ is nested. Moreover, if $G$ is nested with chain of centers $G = X_0 > X_1 > \dotsb > X_n \ge 1$, then $\delta_i = X_{i-1}$ for every integer $1 \le i \le n+1$.
\end{thm}

\begin{proof}
First, assume that $G$ is nested with chain of centers $G = X_0 > X_1 > \dotsb > X_n \ge 1$. We now show that  $\delta_i = X_{i-1}$ by induction on $i$.  The case $i = 1$ follows from the definitions since $X_0 = G = \delta_1$. Now assume that  $\delta_i = X_{i-1}$ for some integer $i\ge 1$. Then $\displaystyle \delta_{i+1} = \!\!\!\!\!\!\! \prod_{\chi \in \irr {G \mid [\delta_{i},G]}} \!\!\!\!\!\!\!\!\! Z(\chi)$.  Observe that if $\chi \in \irr {G \mid [\delta_i,G]}$, then $[\delta_i,G] = [X_{i-1},G] \not\le \ker (\chi)$, and thus, $X_{i-1} \not\le Z(\chi)$.  Since $G$ is nested, this implies that the centers of the characters $\chi$ that we are taking the product over will run through $X_i, \dots, X_n$, and so, $\delta_{i+1} =  X_{i} X_{i-1} \dotsb X_n = X_{i}$.  This proves the claim.  In particular, $\delta_{n+1} = X_n$ and $[X_n,G] = 1$ by Lemma \ref{lewisgvz}.  It follows that $\displaystyle \delta_{n+2} = \!\!\!\!\!\!\! \prod_{\chi \in \irr {G \mid 1}} \!\!\!\!\!\!\!\!\! Z(\chi)$.  However, since the set $\irr {G \mid 1}$ is empty, we obtain $X_{n+2} = 1$, and it is easy to see this implies that $\delta_{\infty} = 1$.
	
Now assume that $\delta_\infty = 1$.  By Lemma \ref{nested basics}, we see that $G$ is nested.
\end{proof}


We also define the subgroups $\epsilon_i$ by $\epsilon_1 = G$ and $\epsilon_{i+1} = V (G \mid [\epsilon_i,G])$.  Observe that $\delta_i = \epsilon_i$ when $G$ is a GVZ-group since $V (\chi) = Z (\chi)$ for all characters $\chi \in \irr G$. We have already seen that $\delta_{i+1} \le \delta_i$, so the $\epsilon_i$ is a chain when $G$ is a GVZ-group.  We now show that this is always the case.

\begin{lem} \label{containment}
Let $G$ be a group.  For each $i \ge 1$, $\delta_{i+1} \le \epsilon_{i+1} \le \epsilon_i$.
\end{lem}

\begin{proof}
We first verify $\epsilon_{i+1} \le \epsilon_i$ by induction on $i$. If $i = 1$, it is clear. So let $i \ge 2$, and assume that $\epsilon_{i+1} \le \epsilon_{i}$. Let $\chi \in \irr {G \mid [\epsilon_{i+1},G]}$; then $\chi \in \irr {G \mid [\epsilon_{i},G]}$ by Lemma~\ref{Vprops} and so $\chi$ vanishes off $V (G \mid [\epsilon_{i},G]) = \epsilon_{i+1}$. This forces the relation $V (G \mid [\epsilon_{i+1},G]) \le \epsilon_{i+1}$, as required.

We now verify $\delta_{i+1} \le \epsilon_{i+1}$ by also induction on $i$.  Observe that $\delta_1 = \epsilon_1 = G$.  We may assume that $\delta_i \le \epsilon_i$ for some $i \ge 1$.  It follows that $[\delta_i,G] \le [\epsilon_i,G]$.  Hence, if $\chi \in \irr {G \mid [\delta_i,G]}$, then $\chi \in \irr {G \mid [\epsilon_i,G]}$.  This implies that $Z(\chi) \le V (\chi) \le V (G \mid [\epsilon_i,G]) = \epsilon_{i+1}$.  It follows that $\delta_{i+1} = \prod_{\chi \in [\delta_i,G]} Z(\chi) \le \epsilon_{i+1}$.
\end{proof}

Observe that $\epsilon_i/\epsilon_{i+1} \le {Z} (G/\epsilon_{i+1})$, for each $i\ge 1$, since $[\epsilon_i,G] \le {V} (G \mid [\epsilon_i,G]) = \epsilon_{i+1}$. Also observe that if $\epsilon_i > {Z} (G)$ for some $i$, then $[\epsilon_i, G] > 1$. Therefore the set $\irr {G\mid [\epsilon_i,G]}$ is nonempty, and it follows that $\epsilon_{i+1} \ge {Z} (G)$.  On the other hand, we see that if $\epsilon_i = Z (G)$, then $[\epsilon_i,G] = 1$ and so $\epsilon_{i+1} = V (G \mid 1) = 1$.  Note that if $[\epsilon_{i-1},G] = [\epsilon_i,G]$, then $\epsilon_{i+1} = V (G \mid [\epsilon_i,G]) = V (G \mid [\epsilon_{i-1},G]) = \epsilon_i$.  We write $\epsilon_\infty$ for the terminal term of this series.

\begin{lem} \label{Z(chi) = V(chi)}
Let $G$ be a group.  If $\chi \in \irr {G \mid U(G)}$, then $Z (\chi) = V (\chi)$.  In particular, if $\chi \in \irr {G \mid U_\infty}$, then $Z (\chi) = V (\chi)$.  
\end{lem}

\begin{proof}
Suppose $\chi \in \irr {G \mid U (G)}$.  Notice that $U (G) \not\le \ker (\chi)$ implies that $\ker (\chi) \le Z (G)$ and $Z(G) = Z(\chi)$ by Lemma \ref{centerproperty}.  By Lemma \ref{uproperties}, we see that $\chi$ vanishes off of $Z(G) = Z(\chi)$.  This implies $V (\chi) = Z (\chi)$.  Now, suppose that $\chi \in \irr {G \mid U_\infty}$.  Thus, we can find an integer $i$ so that $\chi \in \irr {G \mid U_i}$ and $\chi \not\in \irr {G \mid U_{i-1}}$.  This implies that $\chi \in \irr {G/U_{i-1}}$.  Hence, we may assume that $U_{i-1} = 1$.  This implies that $i = 1$, and we are done by the first conclusion.  
\end{proof}

\begin{lem}\label{ind cond}
Let $G$ be a group.  If $[\epsilon_{i+1},G] < [\epsilon_i,G]$, then $\delta_{i+1} = \epsilon_{i+1}$ and $[\epsilon_i,G]/[\epsilon_{i+1},G] \le U (G/[\epsilon_{i+1},G])$.
\end{lem}

\begin{proof}
We work by induction on $i$ to prove $\delta_{i+1} = \epsilon_{i+1}$.  Notice that $\delta_1 = \epsilon_1 = G$.  Suppose $1 \le j < i$ and suppose that $\delta_j = \epsilon_j$.  Notice that $[\epsilon_{j+1},G] = [\epsilon_j,G]$ implies that $\epsilon_{j+2} = V (G \mid [\epsilon_{j+1},G]) = V (G \mid [\epsilon_j,G]) = \epsilon_{j+1}$.  It is not difficult to see that this would imply $\epsilon_i = \epsilon_{i+1}$ and we would have a contradiction.  Thus, $[\epsilon_{j+1},G] < [\epsilon_j,G]$.  Since $\delta_j = \epsilon_j$, we can apply Lemma \ref{nested basics} (3) to see that $Z_{[\epsilon_{j+1}, G]} = \delta_{j+1}$.  Notice that $\epsilon_{j+1} \le Z_{[\epsilon_{j+1},G]}$, so this implies that $\epsilon_{j+1} \le \delta_{j+1}$.  On the other hand, Lemma \ref{containment} implies that $\delta_{j+1} \le \epsilon_{j+1}$.  This gives the equality $\delta_{j+1} = \epsilon_{j+1}$, and we get the conclusion by taking $j = i$.

We have that $\epsilon_{i+1} = \delta_{i+1} = Z_{[\delta_{i+1},G]} = Z_{[\epsilon_{i+1},G]}$.  We also have $\delta_{i+1} = \epsilon_{i+1} = V (G \mid [\epsilon_i,G])$ by definition.  This implies that $[\epsilon_i,G] \le U (G \mid \delta_{i+1})$.  By definition $U/[\epsilon_{i+1},G] = U (G/[\epsilon_{i+1},G]) = U (G/[\epsilon_{i+1}, G] \mid Z (G/[\epsilon_{i+1}, G])) = U (G/[\epsilon_{i+1},G] \mid \delta_{i+1}/[\epsilon_{i+1},G]) = U(G \mid \delta_{i+1})/[\epsilon_{i+1},G]$ where the last equality follows from Lemma \ref{uquotients}.  Hence, we have $U = U(G \mid \delta_{i+1})$ and $[\epsilon_{i},G] \le U$ as desired.  
\end{proof}

We now prove the parallel result for nested GVZ-groups and the $\epsilon$-series that we proved for nested groups and the $\delta$-series. 

\begin{lem} \label{nested gvz quo}
Let $G$ be a group.  If $[\epsilon_{i+1},G] < [\epsilon_i,G]$, then $G/[\epsilon_{i+1},G]$ is a nested GVZ-group.
\end{lem}

\begin{proof}
We know that $[\epsilon_{i+1},G] = [\delta_{i+1},G]$ by Lemma \ref{ind cond}.  By Lemma \ref{nested basics} (5), this implies that $G/[\epsilon_{i+1},G]$ is nested.  Let $\chi \in \irr {G/[\epsilon_{i+1},G]}$.  Let $j$ be minimal so that $[\epsilon_{j+1},G] \le \ker (\chi)$.  If $j = 0$, then $G' \le \ker (\chi)$ and $G = Z (\chi) = V(\chi)$.  Suppose $j \ge 1$.  Notice that $[\epsilon_j,G]$ will not be in the kernel of $\chi$.  By Lemma \ref{ind cond}, we see that $[\epsilon_j,G]/[\epsilon_{j+1},G] \le U(G/[\epsilon_{j+1},G])$.  Hence, we may then apply Lemma \ref{Z(chi) = V(chi)} to see that $Z (\chi) = V (\chi)$.  This implies that $G/[\epsilon_i,G]$ will be a GVZ-group also as desired.
\end{proof}

We finally come to Theorem \ref{intro epsilon}.

\begin{thm}
Let $G$ be a nonabelian group.  Then $G$ is a nested GVZ-group if and only if $\epsilon_\infty = 1$. Moreover, in the event that $G$ is a nested GVZ-group with chain of centers $G = X_0 > X_1 > \dotsb > X_n > 1$, then $\epsilon_{i+1} = X_i$ for every $0 \le i \le n$.
\end{thm}

\begin{proof}
If $G$ is a GVZ-group, then so is any epimorphic image of $G$.  We have $V (\chi) = Z (\chi)$ for every character $\chi \in \irr G$. Furthermore, we have from Lemma~\ref{Vprops} that $V (G \mid [\epsilon_i,G]) = \prod_{\chi} V(\chi)$, where the product is over all characters $\chi \in \irr {G \mid [\epsilon_i,G]}$.  Hence, $\epsilon_i = \delta_i$ for all $i$ when $G$ is a GVZ-group and so the fact that $G$ being a nested GVZ-group implies that $\epsilon_\infty = 1$ follows from Theorem~\ref{delta series}. If $G$ is a nested GVZ-group with chain of centers $G = X_0 > X_1 >\dotsb > X_n > 1$, then Theorem~\ref{delta series} also yields $\epsilon_{i+1} = \delta_{i+1} = X_i$ for every $0 \le i\le n$.  Conversely, if $\epsilon_\infty = 1$, then by Lemma \ref{nested gvz quo}, we see that $G$ is a nested GVZ-group. 
\end{proof}

\section{Examples}

In \cite{ML19gvz}, the second author presents examples of nested GVZ-groups of arbitrarily large nilpotence class.  Obviously, these groups provide examples where $U(G) > 1$ and $Z_{U (G)} > Z(G)$.  In \cite{ML19gvz}, the second author also showed that groups of maximal class are nested, and if they are of class greater than $2$, they are not GVZ-groups.  In particular, when $G$ is a group of order $p^4$ and nilpotence class $3$, then one can see that $G$ is of maximal class, so $K (G) = Z(G) > 1$, but $|G:Z(G)| = p^3$, so no irreducible character is fully-ramified over $Z(G)$, and this implies that $U (G) = 1$.  Similarly, one can show when $G = C_p \wr C_p$ for any odd prime $p$ that $G$ satisfies $K(G) > 1$ and $U(G) = 1$.  


We continue by constructing groups $G$ where $U (G) > 1$ and $U (G) < Z(G)$.  

Let $H$ and $K$ be $p$-groups for some prime $p$ and let $l \ge 1$ be integer.  Suppose that $p^l \le |U (H)|, |U (K)|$ and $p^l < |Z(H)|, |H'|, |Z (K)|, |K'|$.  Fix elements $x_1, \dots, x_l \in U(H)$ and $y_1, \dots, y_l \in U (K)$ be chosen so that $X = \langle x_1, \dots, x_l \rangle$ and $Y = \langle y_1, \dots y_l \rangle$ both have order $p^l$.  (I.e., if we think of $ U(H)$ and $U (K)$ as vector spaces, then the $x_i$'s and the $y_i$'s each form linearly independent subsets.)  We then take $N \le Z (H \times K)$ by $N = \langle (x_1,y_1), \dots, (x_l,y_l) \rangle$.  Let $G = (H \times K)/N$ and observe that $Z (H \times K) = Z(H) \times Z (K)$, so $(Z(H) \times Z(K))/N \le Z(G)$. 

Let $U/N = U(G)$.  Take $M = (X \times Y)/N$.  Observe that $G/M \cong ((H \times K)/N)/((X \times Y)/N) = H/X \times K/Y$.  Since $|X| < |H'|$ and $|Y| < |K'|$, we see that $H/X$ and $K/Y$ are nonabelian, and we have that $U (G/M) = 1$.  Since $UM/M \le U (G/M)$, we have $U \le M$.

Let $\sigma \in \irr H$ and $\tau \in \irr K$.  We can find characters $\mu \in \irr X$ and $\nu \in \irr Y$ so that $\sigma_X = \sigma (1) \mu$ and $\tau_Y = \tau (1) \nu$.  Hence, $(\sigma \times \tau)_M = \sigma (1) \tau (1) (\mu \times \nu)$.  Observe that $N \le \ker (\sigma \times \tau)$ if and only if $\sigma (x_i) \tau (y_i) = \sigma (1) \tau (1)$ for all $i = 1, \dots, l$.  This implies that $\mu (x_i) \nu (y_i) = 1$ for all $i$.   Thus, $\irr G = \{ (\sigma, \tau) \in \irr {H \times K} \mid \mu (x_i) = \overline {\nu} (y_i) {\rm ~for~all~} i \in \{ 1, \dots, n\} \}$.  Notice that $\mu = 1_X$ if and only if $\nu = 1_Y$.  It follows that if $\chi \in \irr {G \mid M/N}$, then $\chi = \sigma \times \tau$ as above where $\mu \ne 1$ and $\nu \ne 1$.  Since $X \le U (H)$ and $Y \le U (K)$, this implies that $\sigma \in \irr {H \mid U(H)}$ and $\tau \in \irr {K \mid (K)}$.  We deduce that $\sigma$ is fully-ramified with respect to $H/Z(H)$ and $\tau$ is fully-ramified with respect to $K/Z(K)$.  A quick check of degrees reveals that $\chi$ will be fully-ramified with respect to $(H \times K)/(Z(H) \times Z(K))$.  Notice that one consequence of this is that $Z(G) \le (Z(H) \times Z(K))/N$ and so, $Z (G) = (Z(H) \times Z(K))/N$.  A second consequence is that every character in $\irr {G \mid M/N}$ is fully-ramified with respect to $G/Z(G)$.  We conclude that $U (G) = M/N$.  Observe that $|U (G)| = |M:N| = p^{2l}/p^l = p^l$.  Also, since $p^l < |Z(H)|, |Z(K)|$, we have that $U (G) < Z(G)$. 

Next, we show that these examples can be extended to find groups where $Z_{U(G)} = Z(G)$.  To do this, we need to add the assumption that $p^l < |U(H)|$ and $p^l < |U(K)|$.  Thus, we can find $u \in U(H) \setminus X$ and $v \in U(K) \setminus Y$.  Let $L = (N \times \langle (u,v)\rangle)/N$, and observe that $L$ is not contained $M$.  It is not difficult to see that $G/L \cong (H \times K)/ (N \times \langle (u,v)\rangle)$, and by the last paragraph we see that $U (G/L) = (X \langle u \rangle \times Y \langle v \rangle)/L = (M \times \langle u \rangle \times \langle v \rangle)/L$ whereas $UL/L = ML/L = (M \times \langle (u,v) \rangle)/L$.  This implies that $|U(G/L):M/L| = p^2$ and $|UL/L:M/L| = p$ and so, $UL/L < U(G/L)$.  By Lemma \ref{U equiv}, this implies that $Z_{U(G)} = Z(G)$.

Now, we take $H$ and $K$ to be semi-extraspecial groups, then $H \times K$ has nilpotence class $2$, so $G$ has nilpotence class $2$.  We will have that $U (G) = K (G)$, so we also obtain an example $K (G) < Z(G)$, and we will have examples where $Z_{K(G)} = Z(G)$.  Note that we can find semi-extraspecial groups so that $p^l < |U(H)| = |Z(H)|$ and $p^l < |U(K)| = |Z(K)|$ for every prime $p$ and every positive integer $l$, so we can find groups where $|U(G)| = p^l$ and $Z_{U(G)} = Z(G)$ for all primes $p$ and all positive integers $l$.

Finally, using Magma \cite{magma}, we have an examples of a $p$-groups $G$ satisfying $1 < U(G) < K(G)$.

\begin{verbatim}
p1 := PCGroup ([ 8, -3, 3, 3, 3, -3, 3, -3, 3, 2641, 52706, 3970, 
16419, 15851, 1267, 61564, 59412, 5205, 8237 ]);

p2 := PCGroup ([ 8, -5, 5, 5, 5, -5, 5, -5, 5, 1875000, 1790081, 
600009, 2175602, 1275130, 345018, 7504003, 68811, 1779, 84027, 
15025004, 2005012, 251020, 15228, 15150005, 1530013, 726021, 
109229 ]);
\end{verbatim}

Each of these groups has order $p^8$ and nilpotence class $3$.  In each group the center equals $K$ and has order $p^2$, but $U$ has order $p$.  The first group is a $3$-group and the second group is a $5$-group.  We also have found examples for $p = 7, 11, 13, 17, 19$, but we have not included these groups.

\end{document}